\documentclass[11pt]{amsart}

\usepackage{latexsym,rawfonts}
\usepackage{amsfonts,amssymb}
\usepackage{amsmath,amsthm}

\usepackage[plainpages=false]{hyperref}
\usepackage{graphicx}
\usepackage{color}
\usepackage[table]{xcolor}
\usepackage{longtable}

\newcommand{\rn}{\mathbb R}
\newcommand{\rnnn}{\mathbb R^{n+1}}
\newcommand{\ry}{\mathbb R^{3}}

\newcommand{\sn}{ {\mathbb{S}^{n}}}

\newcommand{\ti}{\tilde}

\newcommand{\bbS}{\mathbb{S}}

\newcommand{\al}{\alpha}

\newcommand{\de}{\delta}

\newcommand{\la}{\lambda}

\newcommand{\si}{\sigma}

\renewcommand{\(}{\left(}
\renewcommand{\)}{\right)}

\newcommand{\psum}{{+_{\negthinspace\kern-2pt p}}\,}
\newcommand{\qsum}[1]{{+_{\negthinspace\kern-2pt #1}}\,}
\newcommand{\dpsum}{{\tilde+_{\negthinspace\kern-1pt p}}\,}
\newcommand{\dqsum}[1]{{\tilde+_{\negthinspace\kern-1pt #1}}\,}
\newcommand{\lsub}[1]{\hskip -1.5pt\lower.5ex\hbox{$_{#1}$}}

\numberwithin{equation}{section}

\newtheorem{theo}{Theorem}[section]

\newtheorem{lem}[theo]{Lemma}
\newtheorem{prop}[theo]{Proposition}

\newtheorem{rem}[theo]{Remark} \theoremstyle{definition}
\newtheorem{defi}[theo]{Definition}

\begin{document}

\title{The $L_{p}$  dual Christoffel-Minkowski problem for $1<p<q\leq  k+1$ with $1\leq k\leq n$
}
\author[C. Cabezas]{Carlos Cabezas-Moreno}
\address{Institut f\"{u}r Diskrete Mathematik und Geometrie, Technische Universit\"{a}t Wien, Wiedner Hauptstrasse 8-10, 1040 Wien, Austria
 }
\email{carlos.moreno@tuwien.ac.at}

\author[J. Hu]{Jinrong Hu}
\address{Institut f\"{u}r Diskrete Mathematik und Geometrie, Technische Universit\"{a}t Wien, Wiedner Hauptstrasse 8-10, 1040 Wien, Austria
 }
\email{jinrong.hu@tuwien.ac.at, Hu\_jinrong097@163.com}

\begin{abstract}
In this paper, we investigate an $L_{p}$ Christoffel-Minkowski-type problem that prescribes a class of $L_p$ geometric measures, which are mixtures of the $k$-th area measure and the $q$-th dual curvature measure. By establishing a gradient estimate, we obtain the existence of an even, smooth, strictly convex solution to this problem for $1 < p < q \leq k + 1$, where $1 \leq k \leq n$ and $n \geq 1$.
\end{abstract}


\thanks{This work was supported by the Austrian Science Fund (FWF): Project P36545.}

\maketitle

\baselineskip18pt

\parskip3pt

\section{Introduction}

The $L_p$ dual Christoffel-Minkowski problem involves finding a convex body that satisfies the following fully nonlinear partial differential equation:
\begin{equation}\label{PQ}
\sigma_{k}(\nabla^{2}h + hI) = f h^{p-1} (h^{2} + |\nabla h|^{2})^{\frac{k+1-q}{2}}, \ {\rm on} \ \ \sn,
\end{equation}
where $1 \leq k \leq n$ is an integer, $\sigma_k(\nabla^{2}h + hI)$ is the $k$-th elementary symmetric function of the principal radii of curvature, $\nabla h$ is the standard spherical gradient of $h$, $\nabla^{2}h+hI$ is the spherical Hessian of $h$ with respect to a local orthonormal frame on $\sn$, and $I$ is the identity matrix. Equation \eqref{PQ} encompasses a variety of well-known Minkowski-type problems.

When $p=1$ and $q=k+1$, equation \eqref{PQ} corresponds to the Christoffel-Minkowski problem of characterizing the $k$-th area measure. The study of the Christoffel problem for $k=1$ was initiated by Christoffel \cite{Cr85}, further developed by  Berg \cite{Be96} and Firey \cite{F67,F68}. For $k=n$, the existence, regularity, and uniqueness of the Minkowski problem were investigated by Minkowski \cite{M897,M903}, Aleksandrov \cite{A39,A42}, Nirenberg \cite{N53}, Pogorelov \cite{P78}, Cheng-Yau \cite{CY76}, and others. In the intermediate case $1<k<n$, the Christoffel-Minkowski problem has been extensively studied, with major contributions from Firey \cite{Fir70}, Guan-Ma \cite{GM03}, Guan-Lin-Ma \cite{GLM06}, Guan-Ma-Zhou \cite{GMZ06}, and Bryan-Ivaki-Scheuer \cite{BIS23a}.

In the case of $q = k + 1$, equation \eqref{PQ} reduces to the $L_p$ Christoffel-Minkowski problem. When $k = n$, the $L_p$ Minkowski problem was first formulated by Lutwak \cite{L93} and solved by Lutwak and Oliker in the even regular case \cite{LO95}. Since then, the study of this problem has undergone rapid development, as demonstrated by a substantial body of research, including \cite{B19,BLYZ12,B17,CL17,CW06,LW13,LYZ04,Zhu14,Zh15,Zhu15}. For $1 < k < n$, the $L_p$ Christoffel-Minkowski problem has been studied by Hu-Ma-Shen \cite{HMS24} for the case $p > k + 1$ (see also Ivaki \cite{Iva19} and Sheng-Yi \cite{SY20}), while Guan-Xia \cite{GX18} studied the case  $1 < p < k + 1$.

For $k = n$, equation \eqref{PQ} corresponds to the $L_p$ dual Minkowski problem of prescribing the $(p,q)$-th dual curvature measure, first introduced by Lutwak-Yang-Zhang \cite{LYZ18}. The existence of solutions has been extensively investigated. Utilizing a variational approach, Huang-Zhao \cite{HZ18} established existence results for $p > 0$, $q < 0$, for $p, q > 0$, and for $p, q < 0$ under the assumption $p \neq q$ with symmetry conditions. B\"{o}r\"{o}czky-Fodor \cite{BF19} extended these results from the symmetric setting to the general case for $p > 1$, $q > 0$, and $p \neq q$. Further generalizations were obtained by Chen-Huang-Zhao \cite{CHZ19} for symmetric solutions when $pq \geq 0$, and by Chen-Li \cite{CL21}, who employed Gauss curvature flows to establish existence for a broader range with $p > 0, q \neq n + 1$ or $q\leq p <0$.

A fundamental challenge in the $L_p$ dual Minkowski problem lies in establishing solvability for $p \leq 0$ and $q > 0$. This range of $p,q$ encompasses both the logarithmic Minkowski problem ($p = 0$, $q = n + 1$) and the centro-affine Minkowski problem ($p = -n - 1$, $q = n + 1$). In this setting, the case $q > 0$ and $-q^* < p < 0$ was examined by Chen-Chen-Li \cite{CCL21} under symmetry assumptions and was further extended by B\"{o}r\"{o}czky-Kov\'{a}cs-Mui-Zhang \cite{BKMZ24} in a more general framework. Mui \cite{Mu24} investigated the case $q < p + 1$, $-1 < p < 0$, and $p \neq q$ for symmetric solutions.

Regarding uniqueness of the $L_{p}$ dual Minkowski problem, Huang-Zhao \cite{HZ18} established results for $p > q$, while Li-Liu-Lu \cite{LLJ22} provided counterexamples demonstrating nonuniqueness for $p < 0 < q$. However, the uniqueness question remains largely unresolved in the $p < q$ range.

A particularly interesting special case arises when $p = 0$ and $k=n$, where the $L_p$ dual Minkowski problem reduces to the dual Minkowski problem. This problem was initially introduced and studied by Huang-Lutwak-Yang-Zhang \cite{HLYZ16} for symmetric convex bodies in the range $0 < q < n + 1$. Subsequent developments have led to a rich theory known as the dual Brunn-Minkowski theory, yielding numerous significant results and applications in \cite{BLYZ19,JW17,LSW20,Z17,Z18} and references therein.

The uniqueness of solutions to the dual Minkowski problem has also been widely studied. Zhao \cite{Z17} established uniqueness for $q < 0$, while the case $q = 0$ follows from Aleksandrov's classical uniqueness theorem for integral curvature. For the logarithmic Minkowski problem ($q = n + 1$), various uniqueness results have been obtained: B\"{o}r\"{o}czky-Lutwak-Yang-Zhang \cite{BLYZ13} proved uniqueness for even measures when $n = 1$, while Chen-Huang-Li-Liu \cite{CY20}, building upon local results of Kolesnikov-Milman \cite{KM22}, established uniqueness for densities $f$ close to 1 in the $C^{\alpha}$ norm under symmetry assumptions with $0<\alpha<1$. This result was extended by Chen-Feng-Liu \cite{CFL22} to the non-symmetric setting in $\ry$. More recently, B\"{o}r\"{o}czky-Saroglou \cite{BS24} and Hu-Ivaki \cite{HI25} independently extended the result of Chen-Feng-Liu to higher dimensions using different approaches, while Ivaki-Milman \cite{IM25} established uniqueness in the even logarithmic Minkowski problem under curvature pinching conditions.

Despite these advances, the uniqueness of the dual Minkowski problem for $0 < q \neq n + 1$ remains largely open. Recently, Hu \cite{HJ24} obtained the related uniqueness results of symmetric solutions for densities $f$ close to 1 in the $C^{\alpha}$ norm provided $0 < q \leq  n$ for $1 \leq n \leq 3$ or $n - 3 \leq q \leq n$ for $n > 3$.

Concerning the existence and uniqueness of solutions to the $L_p$ dual Christoffel-Minkowski problem \eqref{PQ} for the general case $1 \leq k < n$, several results have been established.  Ding-Li \cite{DL23} addressed the cases $p > q$ with $p > 1$ and $q < p < 0$ by virtue of flow methods. Meanwhile, Chen-Tu-Xiang \cite{CTX25} settled the case $p \geq q$ with $p \geq 1$. To our knowledge, the case where $p < q$ remains unexplored.

In this paper, by establishing a gradient estimate (see Lemma \ref{Wei}) and a full rank theorem (cf. Theorem \ref{Sph}), we obtain the following existence result of \eqref{PQ} for $1<p<q\leq k+1$ through degree theory.
\begin{theo}\label{Thm1}
Let $1\leq k<n$ and $1<p<q\leq k+1$. Let $f$ be a positive, even and smooth function satisfying
\begin{equation}\label{f1}
\nabla^{2}\left( f^{-\frac{1}{k+p-1}}\right)+f^{-\frac{1}{k+p-1}}I\geq 0.
\end{equation}
Then a smooth, origin-symmetric and strictly convex solution exists to Eq. \eqref{PQ}.
\end{theo}

As previously mentioned, when $k=n$, \eqref{PQ} corresponds to the solvability of the $L_{p}$ dual Minkowski problem:
\begin{equation}\label{PQ2}
\sigma_{n}(\nabla^{2}h+hI)=fh^{p-1}(h^{2}+|\nabla h|^{2})^{\frac{n+1-q}{2}}, \quad {\rm on}\ \sn.
\end{equation}
The existence and uniqueness of solutions to the even $L_{p}$ dual Minkowski problem when $p<q$ can also be derived.

\begin{theo}\label{coro2}
 Let $1<p<q\leq n+1$.   Let $f$ be an even, smooth and positive function on $\sn$. Then  Eq. \eqref{PQ2} has a smooth, origin-symmetric and strictly convex solution. Furthermore, there exists a constant $\varepsilon_{0}>0$ depending only on $n, \alpha $ such that if $||f-1||_{C^{\alpha}}\leq \varepsilon$ for some small $\varepsilon\in (0,\varepsilon_{0})$ and $\alpha\in (0,1)$, then the convex solution is unique.
\end{theo}

Regarding the proof of Theorem \ref{coro2}, we note that after obtaining the $C^0$ estimate in Lemma \ref{Bou} for $k = n$ by establishing a gradient estimate, the rest of the proof follows the argument in the proof of \cite[Theorem 1.3]{HJ24}, which establishes the existence and uniqueness of solutions to the dual Minkowski problem for positive indices using degree theory. A key step of applying degree theory is that the linearized operator of \eqref{PQ2} at $h = 1$ is $\Delta + q - p$, where $\Delta$ is the Beltrami-Laplace operator on $\mathbb{S}^n$. This operator is invertible and has exactly one positive eigenvalue, $(q - p)$, with multiplicity 1, in the case $1 < p < q \leq n+1$.

The paper is organized as follows. In Section \ref{Sec2}, we present some basic properties related to convex bodies and $k$-th elementary functions. In Section \ref{Sec3}, we provide the proof of Theorem \ref{Thm1}.  Finally, in Section \ref{Sec4}, we derive a new uniqueness result for solutions to \eqref{PQ} with $f = 1$ and $p>1-k$.

\section{Preliminaries}
\label{Sec2}
In this section, we collect some basic facts about convex bodies and $k$-th elementary functions.
\subsection{Basics on convex bodies}

The study of convex bodies is extensively covered in many standard references, including the comprehensive books of  Gardner \cite{G06} and Schneider \cite{S14}.

Let ${\rnnn}$ standard for the $(n+1)$-dimensional Euclidean space.   For any two vectors $Y, Z\in {\rnnn}$, the standard inner product is expressed as $ \langle Y, Z\rangle $. The Euclidean norm of a vector $X\in{\rnnn}$ is $|X|=\sqrt{\langle X, X\rangle}$ and ${\sn}$ denotes the unit sphere.  A convex body is characterized as a compact convex set of ${\rnnn}$ with a non-empty interior.

For a convex body $K$, the support function in $\rnnn$ is defined for $x\in{\sn}$ as
\[
h_{K}(x)=\max\{\langle x, Y\rangle:Y \in K\}.
\]
Unless it leads to ambiguity, $h_{K}$ is shortened to $h$.

 The radial function $\rho_{K}$ of $K$ is denoted by
\begin{equation*}
\rho_{K}(u)=\max\{s>0: su\in K\}, \quad \forall u\in \sn.
\end{equation*}
Observe that $\rho_{K}(u)u\in \partial K$ for any $u\in \sn$. Similarly, $\rho_{K}$ is shortened to $\rho$ unless ambiguity arises.

For a convex body $K$ in $\rnnn$, the unit outer normal at $X\in \partial K$ is uniquely determined for $\mathcal{H}^{n}$ almost all points $X$. In such case,  the Gauss map $\nu_{K}$ assigns each $X\in \partial K$ to its unique unit outer normal.
For a subset $\omega\subset {\sn}$, the inverse Gauss map $\nu_{K}$ is described as
\begin{equation*}
\nu^{-1}_{K}(\omega)=\{X\in \partial K:  \nu_{K}(X) {\rm \ is \ defined \ and }\ \nu_{K}(X)\in \omega\}.
\end{equation*}

For a smooth and strictly convex body $K$, implying that its  boundary is of class $C^{\infty}$ and of positive Gauss curvature, the support function of $K$ can be written as
\begin{equation*}
h(x)=\langle x,\nu^{-1}_{K}(x)\rangle=\langle\nu_{K}(X),X\rangle,
\end{equation*}
where $x\in {\sn}, \ \nu_{K}(X)=x \ {\rm and} \ X\in \partial K$.

 Consider a local orthonormal basis $\{e_{1},e_{2},\ldots, e_{n}\}$ on $\sn$. Let $h_{i}$ and $h_{ij}$ represent the first and second-order covariant derivatives of $h$ on ${\sn}$ with respect to the frame, then the inverse Gauss map can be formulated as
\begin{equation*}
\nu^{-1}_{K}(x)=Dh(x)=\sum_{i} h_{i}e_{i}+h_{K}(x)x=\nabla h(x)+h(x)x,
\end{equation*}
where $D h$ is the Euclidean gradient on $\rnnn$. The connection between $u$ and $x$ is established by $\rho(u)u=\nabla h(x)+h(x)x=Dh(x)$.

Let  $\lambda=(\lambda_{1},\ldots, \lambda_{n})$ signify the principal radii of curvature at the point $X(x)\in \partial K$, which is the eigenvalues of matrix $\{h_{ij}+h\delta_{ij}\}$.  The Gauss curvature $\kappa$  of $\partial K$ is given by
\begin{equation*}
\kappa=\frac{1}{\sigma_{n}(h_{ij}+h\delta_{ij})}=\frac{1}{\det(h_{ij}+h\delta_{ij})}.
\end{equation*}

\subsection{$k$-th elementary symmetric function.} Let $P=\{a_{ij}\}$ be an $n\times n$ symmetric matrix,
\[
\sigma_{k}(P)=\sigma_{k}(\Lambda(P))=\sum_{1\leq i_{1}<i_{2}\ldots <i_{k}\leq n}\Lambda_{i_{1}}\Lambda_{i_{2}}\ldots \Lambda_{i_{k}},
\]
where $\Lambda:=\Lambda(P)=(\Lambda_{1},\ldots,\Lambda_{n})\in \mathbb{R}^{n}$ is the set of eigenvalues of $P$.
\begin{defi}
Let $1\leq k\leq n$ and $\Gamma_{k}$ be a cone in $\mathbb{R}^{n}$ defined as
\[
\Gamma_{k}=\{\Lambda \in \mathbb{R}^{n}:\sigma_{i}(\Lambda)>0, \ \forall 1\leq i \leq k\}.
\]
\end{defi}

\begin{defi}
A function $h\in C^{2}(\sn)$ of \eqref{PQ} is called $k$-admissible solution if
\[
b=\nabla^{2}h(x)+h(x)I\in \Gamma_{k}
\]
for all points $x\in \sn$. If $b\in \Gamma_{n}$, then $h$ is strictly (spherical) convex.
\end{defi}
We denote by $\sigma_{k}(\Lambda|i)$ the symmetric function with $\Lambda_{i}=0$.  Some basic properties of $k$-th elementary symmetric functions are given that we shall use in what follows.
\begin{prop}
Let $\Lambda=(\Lambda_{1},\cdots, \Lambda_{n})\in \mathbb{R}^{n}$ and $k=0,1,\cdots, n$. Then,
\begin{enumerate}
\item[(i)] $\sigma_{k+1}(\Lambda)=\sigma_{k+1}(\Lambda|i)+\lambda_{i}\sigma_{k}(\Lambda|i), \quad \forall 1\leq i\leq n$.
\item[(ii)] $\sum_{i=1}^{n}\Lambda_{i}\sigma_{k}(\Lambda|i)=(k+1)\sigma_{k+1}(\Lambda)$.

\item[(iii)]$\sum_{i=1}^{n}\sigma_{k}(\Lambda|i)=(n-k)\sigma_{k}(\Lambda)$.

\item[(iv)] $\sum^{n}_{i=1}\Lambda^{2}_{i}\sigma_{k-1}(\Lambda|i)=\sigma_{1}(\Lambda)\sigma_{k}(\Lambda)-(k+1)\sigma_{k+1}(\Lambda)$.

\end{enumerate}
\end{prop}

\begin{prop}[Newton-Maclaurin inequality] For $\Lambda\in \Gamma_{k}$ and $1\leq l \leq k\leq n$, we have
\[
\left[\frac{\sigma_{k}(\Lambda)}{\binom{n}{k}}\right]^{\frac{1}{k}}\leq \left[\frac{\sigma_{l}(\Lambda)}{\binom{n}{l}}\right]^{\frac{1}{l}}.
\]
\end{prop}

\section{Existence of solutions}
\label{Sec3}
To tackle the $L_p$ dual Christoffel-Minkowski problem, a key step is obtaining a priori estimates. Regarding the $C^0$ estimate, where $p \geq q$, the maximum principle straightforwardly provides upper and lower bounds for the solutions, as detailed in \cite{CTX25}. However, when $p < q$, the $C^0$ estimate relies on a gradient estimate.

Let $R = \max_{\mathbb{S}^n} h$ and $r = \min_{\mathbb{S}^n} h$. Inspired by \cite[Lemma 2.2]{HI24}, a gradient estimate for solutions to \eqref{PQ} can be derived as follows.

\begin{lem}\label{Wei}
Let $1\leq k\leq n$. Let $1<p\leq k+1$ and $q\leq k+1$.  Suppose $h$ is a positive, smooth and strictly convex solution to Eq. \eqref{PQ}. For any $0<\gamma<2(p-1)/k\leq 2$. Then there exists a constant $M\geq 2$, depending on $\gamma, k, \min_{\sn} f$ and $||f||_{C^{1}(\sn)}$ such that
\begin{equation}\label{mia}
\frac{h^{2}+|\nabla h|^{2}}{h^{\gamma}}\leq M R^{2-\gamma}.
\end{equation}
\end{lem}
\begin{proof}
 Recall $\rho^{2}=h^{2}+|\nabla h|^{2}$. Let $b[h]:=\nabla^{2}h+hI$. We define the function $\zeta=\frac{\rho^{2}}{h^{\gamma}}$ for $0< \gamma <2(p-1)/k$. Set $b=b[h]$. If \eqref{mia} is not true, we may assume $\max_{\sn} \zeta> M R^{2-\gamma}$. Suppose $\zeta$ attains its maximum at a point $x_{0}\in\sn$, then at $x_{0}$, we obtain
\begin{equation}\label{ji}
(\nabla^{2}h+hI)\nabla h=\frac{\gamma}{2}\frac{\rho^{2}}{h}\nabla h.
\end{equation}
 Since $\nabla h \neq 0$ at $x_{0}$, \eqref{ji} implies that at $x_{0}$, $\nabla h$ is an eigenvector of $b$. Hence we may choose an orthonormal basis $\{e_{i}\}$ for $T_{x_{0}}(\sn)$, such that $e_{1}=\frac{\nabla h}{|\nabla h|}$, $b|_{x_{0}}$ is diagonal and $b_{1i}=0$ for $i=2,\ldots, n$, where $T_{x_{0}}(\sn)$ is the tangent space of $\sn$ at $x_{0}$. Together with \eqref{ji}, we have
\begin{equation}\label{aa}
b_{11}=\frac{\gamma}{2}\frac{\rho^{2}}{h}.
\end{equation}
On the other hand, at $x_{0}$, by a direct calculation, we obtain
\begin{equation*}
\begin{split}
\zeta_{;ii}&=\frac{2}{h^{\gamma}}(\sum_{\ell}b_{\ell i;i}h_{\ell}+b^{2}_{ii}-hb_{ii})-4\gamma\frac{\sum_{\ell}b_{\ell i}h_{\ell}h_{i}}{h^{\gamma+1}} -\frac{\gamma \rho^{2}(b_{ii}-h\delta_{ii})}{h^{\gamma+1}}+\gamma(\gamma+1)\frac{\rho^{2}h^{2}_{i}}{h^{\gamma+2}},
\end{split}
\end{equation*}
and
\begin{equation}
\begin{split}
\label{rt}
\frac{\zeta_{;ii}}{\zeta}&=\frac{2}{\rho^{2}}(\sum_{\ell}b_{\ell i;i}h_{\ell}+b^{2}_{ii}-hb_{ii})-4\gamma \frac{b_{ii}h^{2}_{i}}{h\rho^{2}} - \frac{\gamma(b_{ii}-h\delta_{ii})}{h}+\gamma(\gamma+1)\frac{h^{2}_{i}}{h^{2}}.
\end{split}
\end{equation}
Let  $\sigma^{ii}_{k}=\frac{\partial \sigma_{k}}{\partial \lambda_{i}}$. Since  $\zeta_{;ii}\leq 0$ and $\zeta_{i}=0$ for $i=1,\ldots,n$ at $x_{0}$, using \eqref{aa} and \eqref{rt}, we obtain
\begin{equation*}
\begin{split}
0&\geq \frac{2}{\rho^{2}}\sum_{i}\sigma^{ii}_{k}(\sum_{\ell}b_{\ell i;i}h_{\ell}+b^{2}_{ii}-hb_{ii})-4\gamma \sigma^{11}_{k}\frac{b_{11}h^{2}_{1}}{h\rho^{2}}\\
&\quad -\gamma\sum_{i} \sigma^{ii}_{k}\frac{b_{ii}}{h}+\gamma \sum_{i} \sigma^{ii}_{k}+\gamma(\gamma+1)\sigma^{11}_{k}\frac{h^{2}_{1}}{h^{2}}\\
&=\frac{2}{\rho^{2}}\left((h^{p-1}f\rho^{k+1-q})_{1}h_{1}+\sum_{i}\sigma^{ii}_{k}b^{2}_{ii}-kh^{p}f\rho^{k+1-q}\right)-2 \gamma^{2}\sigma^{11}_{k}\frac{h^{2}_{1}}{h^{2}}\\
&\quad -\gamma k h^{p-2}f\rho^{k+1-q}+\gamma \sum_{i} \sigma^{ii}_{k}+\gamma(\gamma+1)\sigma^{11}_{k}\frac{h^{2}_{1}}{h^{2}}.
\end{split}
\end{equation*}
Due to $\rho \rho_{1}=h_{1}b_{11}$ and $\frac{2}{\rho^{2}}\sigma^{11}_{k}b^{2}_{11}=\frac{\rho^{2}\sigma^{11}_{k}\gamma^{2}}{2h^{2}}$ at $x_{0}$, then for $p>1$ and $q\leq k+1$, we have
\begin{equation}
\begin{split}
\label{q3}
0&\geq 2(p-1)\frac{h^{p-2}h^{2}_{1}f\rho^{k+1-q}}{\rho^{2}}+2\frac{h^{p-1}h_{1}f_{1}\rho^{k+1-q}}{\rho^{2}}+2(k+1-q)\frac{h^{p-1}f\rho^{k-q}\rho_{1}h_{1}}{\rho^{2}}\\
&\quad +\frac{2}{\rho^{2}}\sum_{i>1}\sigma^{ii}_{k}b^{2}_{ii}-2k\frac{h^{p}f\rho^{k+1-q}}{\rho^{2}}+\left( \frac{\gamma^{2}\rho^{2}}{2h^{2}_{1}}-\gamma(\gamma-1) \right)\sigma^{11}_{k}\frac{h^{2}_{1}}{h^{2}}-\gamma k h^{p-2}f\rho^{k+1-q}+\gamma \sum_{i}\sigma^{ii}_{k}\\
&\geq 2(p-1)\frac{h^{p-2}h^{2}_{1}f\rho^{k+1-q}}{\rho^{2}}+2\frac{h^{p-1}h_{1}f_{1}\rho^{k+1-q}}{\rho^{2}}-2k\frac{h^{p}f\rho^{k+1-q}}{\rho^{2}}\\
&\quad +\left( \frac{\gamma^{2}\rho^{2}}{2h^{2}_{1}}-\gamma(\gamma-1) \right)\sigma^{11}_{k}\frac{h^{2}_{1}}{h^{2}}-\gamma k h^{p-2}f\rho^{k+1-q}.
\end{split}
\end{equation}
For $0<\gamma \leq 2$, there holds
\begin{equation}\label{q2}
 \frac{\gamma^{2}\rho^{2}}{2h^{2}_{1}}-\gamma(\gamma-1)\geq \frac{\gamma^{2}}{2}-\gamma(\gamma-1)=\gamma\left(1-\frac{\gamma}{2}  \right)\geq 0.
\end{equation}
Moreover, by the assumption, when for some $M \geq 2$, $\frac{\rho^{2}}{h^{\gamma}}\big|_{x_{0}}>M R^{2-\gamma}$, then at $x_{0}$, we find
\begin{equation}\label{q1}
h^{2}_{1}\geq M R^{2-\gamma}h^{\gamma}-h^{2}\geq \frac{M}{2}R^{2-\gamma}h^{\gamma}.
\end{equation}
Substituting \eqref{q1} and \eqref{q2} into \eqref{q3}, for $0<\gamma<2(p-1)/k$, at $x_{0}$, we obtain
\begin{equation}
\begin{split}
\label{q4}
0&\geq \frac{2h^{p-2}f\rho^{k+1-q}}{\rho^{2}}\left( (p-1)h^{2}_{1}+hh_{1}(\log f)_{1}-kh^{2}-\frac{k\gamma}{2}(h^{2}+h^{2}_{1})\right)\\
&=\frac{2h^{p-2}f\rho^{k+1-q}}{\rho^{2}}\left(\left(p-1-\frac{k\gamma}{2} \right)h^{2}_{1}+hh_{1}(\log f)_{1}-\left(k+\frac{k\gamma}{2} \right)h^{2} \right)\\
&\geq \frac{2h^{p-2}f\rho^{k+1-q}}{\rho^{2}}\left(\left(p-1-\frac{k\gamma}{2} \right)\frac{M}{2}R^{2-\gamma}h^{\gamma}-c_{1}M^{\frac{1}{2}}R^{1-\frac{\gamma}{2}}h^{1+\frac{\gamma}{2}}-c_{2}h^{2}\right)\\
&=\frac{2h^{p+\gamma-2}\rho^{k+1-q}f}{\rho^{2}}\left( \left( p-1-\frac{k\gamma}{2}\right)\frac{M}{2}R^{2-\gamma}-c_{1}M^{\frac{1}{2}} R^{1-\frac{\gamma}{2}}h^{1-\frac{\gamma}{2}}-c_{2}h^{2-\gamma}\right)\\
&\geq \frac{2h^{p+\gamma-2}\rho^{k+1-q}R^{2-\gamma}f}{\rho^{2}}\left(\left( p-1-\frac{k\gamma}{2} \right)\frac{M}{2}-c_{1}M^{\frac{1}{2}}-c_{2} \right),
\end{split}
\end{equation}
 where the positive constant $c_{1}$ depends on $\min_{\sn}f$ and $||f||_{C^{1}(\sn)}$,  and the positive constant $c_{2}$ depends on $k$. However, when $M$ is sufficiently large, which contracts to \eqref{q4}. The proof is completed.
\end{proof}

The following lemma due to Guan \cite[Lemma 3.1]{Gu23} (see also Mei-Wang-Weng \cite[Lemma 4.2]{MWW24}) plays an important role in obtaining the uniform $C^{0}$ and $C^{1}$ estimates of solutions to \eqref{PQ}.
\begin{lem}\label{COL}
If $h$ is an even, smooth and strictly convex function on $\sn$ satisfying
\begin{equation*}
\frac{|\nabla h|^{2}}{h^{\gamma}}\leq M R^{2-\gamma}, \quad \ {\rm on} \ \ \sn
\end{equation*}
for some $M>0$, $\gamma>0$. Then the following non-collapsing estimate holds
\begin{equation*}
\frac{R}{r}\leq C,
\end{equation*}
where the constant $C$ depends only on $n$, $M$ and $\gamma$.
\end{lem}

Utilizing Lemma \ref{Wei} and Lemma \ref{COL}, we can derive the uniform estimates for the solution $h$ to \eqref{PQ} as below.
\begin{lem}\label{Bou}
Let $1\leq k \leq n$ and $1<p<q\leq k+1$. Suppose $h$ is an even, smooth and strictly convex solution to Eq. \eqref{PQ}. Then we have
\begin{equation}\label{C0}
1/C\leq h \leq C,
\end{equation}
and
\begin{equation}\label{C1}
 \quad |\nabla h|\leq C,
\end{equation}
where $C>0$ is a constant depending only on $n,\gamma,k, \min_{\sn} f$ and $||f||_{C^{2}(\sn)}$.
\end{lem}
\begin{proof}
From \eqref{PQ}, we obtain the following inequalities:
\begin{equation}\label{oi}
R^{q-p} \geq c_{n,k} \min_{\mathbb{S}^n} f, \quad r^{q-p} \leq c_{n,k} \max_{\mathbb{S}^n} f,
\end{equation}
where the positive constant $c_{n,k}$ depends only on $n$ and $k$. If $p < q$, from \eqref{oi}, we can derive a uniform upper bound for $r$ and a uniform lower bound for $R$. These facts, together with Lemmas \ref{Wei} and \ref{COL}, yield \eqref{C0}. Equation \eqref{C1} follows directly from \eqref{C0} and Lemma \ref{Wei}.
\end{proof}

The following full rank theorem ensures the strict convexity of solutions to \eqref{PQ} when $1\leq k<n$.
\begin{theo}\label{Sph}
Let $1\leq k<n$. Let $p\geq 1$ and $q\in \rn$. Assume that $h$ is a positive and smooth  solution of Eq. \eqref{PQ} with $\nabla^{2}h+hI\geq 0 $ on $\sn$.  Let $f$ be a positive and smooth function satisfying the condition \eqref{f1}.
Then $\nabla^{2}h+hI$ is positive definite on $\sn$.
\end{theo}
\begin{proof}  We follow the approach in \cite{BIS23b}. Let us set
\begin{equation*}
G=\si_k^\frac{1}{k}=h^\frac{p-1}{k} \rho^\frac{k+1-q}{k} f^\frac{1}{k}=h^{\tilde{\alpha}}\rho^{\tilde{\beta}} \ti f.
\end{equation*}
Fix  $x_{0}\in \sn$. Assume that $\varphi$ is a smooth function such that $\la_1\geq \varphi$ in an open neighborhood around $x_0$ and $\la_1=\varphi$ at $x_0$. Assume $b_{ij}$ is diagonal at $x_0$.  Let $D$ be the multiplicity of the smallest eigenvalue $\la_1$ of $b_{ij}$ at $x_{0}$ and arrange the eigenvalues as
\[
\la_1=\ldots=\la_D<\la_{D+1}\leq \ldots \leq \la_n.
\]
In the following, the semi-colon notation denotes the covariant derivative. Due to \cite[Lemma 5]{BCD17}, we have
\begin{equation}
\begin{split}
\label{z1}
b_{kl;i}&=\varphi_{;i} \de_{kl},\quad \text{for all $1\leq k,l\leq D$};\\
\varphi_{;ii}&\leq b_{11;ii}-2\sum_{l>D}\frac{1}{\la_l-\la_1}(b_{1l;i})^2.
\end{split}
\end{equation}
Recall the commutator formula on $\bbS^n$:
\begin{equation}
\begin{split}
\label{z2}
b_{11;ii}=b_{ii;11}-b_{ii}+b_{11}.
\end{split}
\end{equation}
Then using \eqref{z1} and \eqref{z2}, we further have
\begin{equation}
\begin{split}
\label{z3}
G^{ij}\varphi_{;ij}&\leq  \sum_{i}G^{ii}b_{11;ii}-2\sum_{i}\sum_{l>D}\frac{G^{ii}}{\la_l-\la_1}(b_{1l;i})^2\\
&=G_{;11}-\sum_{i,j,k,l}G^{ij,kl}b_{ij;1}b_{kl;1}-G+\la_1 \sum_{i}G^{ii}-2\sum_{i}\sum_{l>D}\frac{G^{ii}}{\la_l-\la_1}(b_{1l;i})^2.
\end{split}
\end{equation}

On the other hand, by the inverse concavity of $G$ (see the proof of \cite[Theorem. 3.4]{BIS23b} for details), we have
\begin{equation}
\begin{split}
\label{z6}
-\sum_{i,j,k,l}G^{ij,kl}b_{ij;1}b_{kl;1}-2\sum_{i}\sum_{l>D}\frac{G^{ii}}{\la_l}(b_{1l;i})^2
\leq &~-\frac{2}{G}|\nabla_1 G|^2+C_1|\nabla \varphi|.
\end{split}
\end{equation}
Here the positive constants $C_{i}$ depend on $f,\dot{G}$, $\ddot{G}$ and $|\nabla b|$.

Applying \eqref{z6} to \eqref{z3}, we obtain
\begin{equation}
\begin{split}
\label{z7}
G^{ij}\varphi_{;ij} &\leq \la_1 \sum_{i} G^{ii}+C_1|\nabla \varphi|+ \(G_{;11}-G-2\frac{|G_{;1}|^2}{G}\).
\end{split}
\end{equation}
By a direct computation,
\begin{equation}
\begin{split}
\label{z9}
(\rho^{\tilde{\beta}})_{;1}&=\tilde{\beta}\rho^{\tilde{\beta}-1}\frac{\la_1 h_1}{\rho},\\
(\rho^{\tilde{\beta}})_{;11}&=\tilde{\beta}\rho^{\tilde{\beta}-1}\rho_{11}+\tilde{\beta}(\tilde{\beta}-1)\rho^{\tilde{\beta}-2}\rho_1^2\\
&=\tilde{\beta} \rho^{\tilde{\beta}-2}\left(\sum_{\ell}h_{\ell}b_{11;\ell}+b^{2}_{11}-hb_{11}-\rho^{2}_{1}\right)+\tilde{\beta}(\tilde{\beta}-1)\rho^{\tilde{\beta}-2}\rho_1^2\\
&=\tilde{\beta}\rho^{\tilde{\beta}-2}\sum_{\ell}h_{\ell}b_{11;\ell}+\tilde{\beta} \rho^{\tilde{\beta}-2}\lambda^{2}_{1}-\tilde{\beta} h\rho^{\tilde{\beta}-2}\lambda_{1}+\tilde{\beta}(\tilde{\beta}-2)\rho^{\tilde{\beta}-4}\la_1^2h_1^2.
\end{split}
\end{equation}
Using \eqref{z9}, for $\tilde{\beta} \in  \rn$,  one sees that
\begin{equation*}
\begin{split}
 &~(h^{\tilde{\al}}\rho^{\tilde{\beta}}\ti f)_{;11}-h^{\tilde{\al}} \rho^{\tilde{\beta}} \ti f-2\frac{(h^{\tilde{\al}} \rho^{\tilde{\beta}} \ti f)_{;1}^2}{h^{\tilde{\al}} \rho^{\tilde{\beta}} \ti f}\\
=&~\((h^{\tilde{\al}}\ti f)_{;11}-h^{\tilde{\al}} \ti f-2\frac{(h^{\tilde{\al}} \ti f)_{;1}^2}{h^{\tilde{\al}}\ti f}\)\rho^{\tilde{\beta}}-2(h^{\tilde{\al}} \ti f)_{;1}(\rho^{\tilde{\beta}})_{;1}\\
&~+h^{\tilde{\al}} \ti f (\rho^{\tilde{\beta}})_{;11}-2h^{\tilde{\al}} \ti f \frac{((\rho^{\tilde{\beta}})_{;1})^2}{\rho^{\tilde{\beta}}}\\
\leq &~\((h^{\tilde{\al}}\ti f)_{;11}-h^{\tilde{\al}} \ti f-2\frac{(h^{\tilde{\al}} \ti f)_{;1}^2}{h^{\tilde{\al}} \ti f}\)\rho^{\tilde{\beta}}+C_{2} \la_1+C_{3}|\nabla \lambda_{1}|.
\end{split}
\end{equation*}

On the other hand, by a direct calculation, it yields
\begin{equation*}
\begin{split}
\label{z10}
 &~(h^{\tilde{\al}}\ti f)_{;11}-h^{\tilde{\al}} \ti f-2\frac{((h^{\tilde{\al}} \ti f)_{;1})^2}{h^{\tilde{\al}} \ti f}\\
=&~(\tilde{\al} h^{\tilde{\al}-1}h_{11}+\tilde{\al}(\tilde{\al}-1)h^{\tilde{\al}-2}h_1^2)\ti f+2\tilde{\al} h^{\tilde{\al}-1}h_1 \ti f_{;1}\\
 &~+h^{\tilde{\al}} \ti f_{;11}-h^{\tilde{\al}} \ti f-2\frac{|\tilde{\al} h^{\tilde{\al}-1}h_1\ti f+h^{\tilde{\al}} \ti f_{;1}|^2}{h^{\tilde{\al}}\ti f} \\
=&~\tilde{\al} h^{\tilde{\al}-1}(\la_1-h)\ti f-\tilde{\al}(\tilde{\al}+1)h^{\tilde{\al}-2}h_1^2\ti f-2\tilde{\al} h^{\tilde{\al}-1}h_1 \ti f_{;1}\\
  &~+h^{\tilde{\al}} \(\ti f_{;11}-\ti f-2 \frac{(\ti f_{;1})^2}{\ti f}\)\\
\leq &~\tilde{\al} h^{\tilde{\al}-1}\ti f \la_1+h^{\tilde{\al}}\(\ti f_{;11}-(\tilde{\al}+1)\ti f-\frac{\tilde{\al}+2}{\tilde{\al}+1} \frac{(\ti f_{;1})^2}{\ti f}\),
\end{split}
\end{equation*}
where we used the Cauchy-Schwarz inequality (assuming $\tilde{\al}\geq 0$, i.e., $p\geq 1$):
\begin{equation}
\begin{split}
\label{z11}
-2\tilde{\al} h^{\tilde{\al}-1}h_1 \ti f_{;1}-\tilde{\al}(\tilde{\al}+1)h^{\tilde{\al}-2}h_1^2 \ti f\leq \frac{\tilde{\al}}{\tilde{\al}+1}h^{\tilde{\al}}\frac{(\ti f_{;1})^2}{\ti f}.
\end{split}
\end{equation}
In addition, for $p \geq 1$, due to  \eqref{f1}, then we have
\begin{equation}
\begin{split}
\label{z12}
&\quad \ti f_{;11}-(\tilde{\al}+1)\ti f-\frac{\tilde{\al}+2}{\tilde{\al}+1} \frac{(\ti f_{;1})^2}{\ti f}\\
&=-(\tilde{\al}+1)\tilde{f}^{\frac{\tilde{\al}+2}{\tilde{\al}+1}}\left((\tilde{f}^{-\frac{1}{1+\tilde{\al}}})_{;11}+\tilde{f}^{-\frac{1}{1+\tilde{\al}}}\right)\\
%&=-(\al+1)\tilde{f}^{\frac{\al+2}{\al+1}}\((f^{-\frac{1}{k+p-1}})_{;11}+f^{-\frac{1}{k+p-1}}\)\\
&=-\frac{p-1+k}{k}\tilde{f}^{\frac{\tilde{\al}+2}{\tilde{\al}+1}}\left((f^{-\frac{1}{k+p-1}})_{;11}+f^{-\frac{1}{k+p-1}}\right)\leq 0.
\end{split}
\end{equation}
Therefore, using \eqref{z12} into \eqref{z7},  in a viscosity sense, we obtain
\begin{equation}
\begin{split}
\label{z13}
G^{ij}\nabla_{ij}^2\la_1\leq C_{4}(\la_1+|\nabla \la_1|).
\end{split}
\end{equation}
If $b_{ij}$ is not full rank, applying the strong maximum principle to \eqref{z13}, we get $\la_1\equiv 0$ on $\sn$. However, at the point where $h$ attains its minimum, we have $\la_1>0$. This is a contradiction.
\end{proof}

To obtain a solution to \eqref{PQ}, it is essential to obtain the uniform upper bound on the principal radii of curvature.  We consider the following general curvature equation:
\begin{equation}\label{cw}
\sigma_{k}(\nabla^{2}h+hI)=f(x,h,\nabla h),\quad {\rm on} \ \sn.
\end{equation}
Inspired by \cite{GRW15, LRW16}, by choosing a suitable test function, which involves the $m$-th polynomial regarding the principal radius of curvature, we can establish the $C^{2}$ estimate for convex solutions to \eqref{cw} as below.
\begin{theo}\label{hc2}
Let $1\leq k <n$. Suppose that $h$ is a positive, smooth and convex solution to Eq. \eqref{cw}. Then there exists a positive constant $C$ such that
\[
\Delta h+nh \leq C,
\]
where $C$ depends on $n,k,\min_{\sn}h,||h||_{C^{1}(\sn)}$, $\min_{\sn}f$ and $||f||_{C^{2}(\sn)}$.
\end{theo}

Before proving Theorem \ref{hc2}, some preparations are necessary. First, we set the test function:
\begin{equation}\label{Q6}
Q=\log P_{m}+\frac{mN}{2}|\nabla h|^{2}+mM\log h,
\end{equation}
where $M, N$ are constants to be determined later, and
\[
P_{m}=\sum_{j}\lambda^{m}_{j}, \quad m\geq 2.
\]
The $\lambda_{1},\lambda_{2}, \ldots, \lambda_{n}$ are the eigenvalues of the spherical Hessian $\nabla^{2}h+hI$. Suppose that the function $Q$ attains its maximum value at some point $x_{0}\in \sn$. By rotating the coordinates, assume that at $x_{0}$, $b_{ij}$ is a diagonal matrix and $\lambda_{1}\geq \lambda_{2}\geq\ldots\geq \lambda_{n}$.

By differentiation \eqref{Q6}, using \cite[Theorem 5.5]{Ba84} (see also \cite[Lemma 2.1]{GRW15}), at $x_{0}$, we obtain
\begin{equation}\label{Q7}
\frac{\sum_{j}\lambda^{m-1}_{j}b_{jj;i}}{P_{m}}+Nh_{i}h_{ii}+M\frac{h_{i}}{h}=0,
\end{equation}
and
\begin{equation}
\begin{split}
\label{bjh}
0\geq &\frac{1}{P_{m}}\left(\sum_{j}\lambda^{m-1}_{j}b_{jj;ii}+(m-1)\sum_{j}\lambda^{m-2}_{j}b^{2}_{jj;i}+\sum_{p,q;p\neq q}\frac{\lambda^{m-1}_{p}-\lambda^{m-1}_{q}}{\lambda_{p}-\lambda_{q}}b^{2}_{pq;i}  \right)\\
&\quad -\frac{m}{P^{2}_{m}}(\sum_{j}\lambda^{m-1}_{j}b_{jj;i})^{2}+\sum_{s}Nh_{s}h_{si;i}+Nh^{2}_{ii}+M\frac{h_{ii}}{h}-M\frac{h^{2}_{i}}{h^{2}}.
\end{split}
\end{equation}
 Differentiating \eqref{cw} twice, at $x_{0}$, we obtain
\begin{equation}\label{bjk}
\sum_{i}\sigma^{ii}_{k}b_{ii;j}=f_{h_{j}}h_{jj}+f_{h}h_{j}+f_{j},
\end{equation}
and
\begin{equation}
\begin{split}
\label{bjl}
\sum_{i}\sigma^{ii}_{k}b_{ii;jj}+\sum_{p,q,r,s}\sigma^{pq,rs}_{k}b_{pq;j}b_{rs;j}&\geq  -C-C |h_{jj}|-C h^{2}_{jj}+\sum_{s}f_{h_{s}}h_{sj;j}\\
&\geq -C-Cb_{jj}-Cb^{2}_{jj}+\sum_{s}f_{h_{s}}b_{sj;j}.
\end{split}
\end{equation}
Here, the positive constants $C$ depend on $\min_{\sn}h, ||h||_{C^{1}(\sn)}, \min_{\sn}f$ and $||f||_{C^{2}(\sn)}$.

Now using the Ricci identity $b_{ii;jj}=b_{jj;ii}+b_{ii}-b_{jj}$, we obtain
\begin{equation}\label{bjp}
\sum_{i}\sigma^{ii}_{k}b_{jj;ii}=\sum_{i}\sigma^{ii}_{k}b_{ii;jj}+b_{jj}\sum_{i}\sigma^{ii}_{k}-kf.
\end{equation}
Multiply both sides of \eqref{bjh} by  $\sigma^{ii}_{k}$, using \eqref{bjk}, \eqref{bjl} and \eqref{bjp} into \eqref{bjh}, we have
\begin{equation}
\begin{split}
\label{zx}
0&\geq \frac{1}{P_{m}}\sum_{j}\lambda^{m-1}_{j}(-C_{0}-C_{1}b_{jj}-C_{2}b^{2}_{jj}+\sum_{s}f_{h_{s}}b_{sj;j}+b_{jj}\sum_{i}\sigma^{ii}_{k}-kf-\sum_{p,q,r,s}\sigma^{pq,rs}_{k}b_{pq;j}b_{rs;j})\\
&\quad +\frac{1}{P_{m}}(m-1)\sum_{i}\sigma^{ii}_{k}\sum_{j}\lambda^{m-2}_{j}b^{2}_{jj;i}+\frac{1}{P_{m}}\sum_{i}\sigma^{ii}_{k}\sum_{p,q;p\neq q}\frac{\lambda^{m-1}_{p}-\lambda^{m-1}_{q}}{\lambda_{p}-\lambda_{q}}b^{2}_{pq;i}\\
&\quad-\sum_{i}\frac{m\sigma^{ii}_{k}}{P^{2}_{m}}(\sum_{j}\lambda^{m-1}_{j}b_{jj;i})^{2}
+\sum_{s,i}N\sigma^{ii}_{k}h_{s}h_{si;i}+N\sum_{i}\sigma^{ii}_{k}h^{2}_{ii}+M\sum_{i}\sigma^{ii}_{k}\frac{h_{ii}}{h}-M\frac{\sum_{i}\sigma^{ii}_{k}h^{2}_{i}}{h^{2}}\\
&= \frac{1}{P_{m}}\sum_{j}\lambda^{m-1}_{j}\left(-C_{0}-C_{1}b_{jj}-C_{2}b^{2}_{jj}-W(\sigma_{k})^{2}_{j}+W(\sigma_{k})^{2}_{j}+\sum_{s}f_{h_{s}}b_{sj;j}\right.\\
&\left.\quad \quad \quad \quad \quad \quad \quad \quad \quad +b_{jj}\sum_{i}\sigma^{ii}_{k}-kf-\sum_{p,q,r,s}\sigma^{pq,rs}_{k}b_{pq;j}b_{rs;j}\right)\\
&\quad +\frac{1}{P_{m}}(m-1)\sum_{i}\sigma^{ii}_{k}\sum_{j}\lambda^{m-2}_{j}b^{2}_{jj;i}+\frac{1}{P_{m}}\sum_{i}\sigma^{ii}_{k}\sum_{p,q;p\neq q}\frac{\lambda^{m-1}_{p}-\lambda^{m-1}_{q}}{\lambda_{p}-\lambda_{q}}b^{2}_{pq;i}\\
&\quad-\sum_{i}\frac{m\sigma^{ii}_{k}}{P^{2}_{m}}(\sum_{j}\lambda^{m-1}_{j}b_{jj;i})^{2}+\sum_{s,i}N\sigma^{ii}_{k}h_{s}b_{si;i}-N\sum_{i}\sigma^{ii}_{k}h^{2}_{i}+N\sum_{i}\sigma^{ii}_{k}b^{2}_{ii}-2Nhkf\\
&\quad+h^{2}N\sum_{i}\sigma^{ii}_{k}+M\frac{kf}{h}-M\sum_{i}\sigma^{ii}_{k}-M\frac{\sum_{i}\sigma^{ii}_{k}h^{2}_{i}}{h^{2}},
\end{split}
\end{equation}
where the positive constants $C_{0}$, $C_{1}$ and $C_{2}$ depend on $\min_{\sn}h,||h||_{C^{1}(\sn)}$, $\min_{\sn}f$ and $||f||_{C^{2}(\sn)}$. Due to \eqref{Q7} and \eqref{bjk}, we have
\begin{equation}
\begin{split}
\label{zq}
&\frac{1}{P_{m}}\sum_{s,j}\lambda^{m-1}_{j}f_{h_{s}}b_{sj;j}+\sum_{s,i}Nh_{s}\sigma^{ii}_{k}b_{si;i}\\
&=-M\sum_{s}f_{h_{s}}\frac{h_{s}}{h}+N\sum_{s}f_{h}h^{2}_{s}+N\sum_{s}f_{s}h_{s}\\
&\geq -\hat{C}.
\end{split}
\end{equation}
Here the positive constant $\hat{C}$ depends on $\min_{\sn}h,||h||_{C^{1}(\sn)}, \min_{\sn}f, ||f||_{C^{2}(\sn)}$ and $ M,N$.
On the other hand,  at $x_{0}$,  there holds
\begin{equation}\label{zr}
-\sum_{p,q,r,s}\sigma^{pq,rs}_{k}b_{pq;j}b_{rs;j}=-\sum_{p,q}\sigma^{pp,qq}_{k}b_{pp;j}b_{qq;j}+\sum_{p,q}\sigma^{pp,qq}_{k}b^{2}_{pq;j}.
\end{equation}
Then substituting \eqref{zq} and \eqref{zr} into \eqref{zx}, we obtain
\begin{equation}
\begin{split}
\label{mj}
0&\geq \frac{1}{P_{m}}\sum_{j}\lambda^{m-1}_{j}(-C_{0}(W)-C_{1}(W)b_{jj}-C_{2}(W)b^{2}_{jj}+W(\sigma_{k})^{2}_{j}-\sum_{p,q}\sigma^{pp,qq}_{k}b_{pp;j}b_{qq;j}+\sum_{p,q}\sigma^{pp,qq}_{k}b^{2}_{pq;j})\\
&\quad+\frac{1}{P_{m}}(m-1)\sum_{i}\sigma^{ii}_{k}\sum_{j}\lambda^{m-2}_{j}b^{2}_{jj;i}+\frac{1}{P_{m}}\sum_{i}\sigma^{ii}_{k}\sum_{p,q; p\neq q} \frac{\lambda^{m-1}_{p}-\lambda^{m-1}_{q}}{\lambda_{p}-\lambda_{q}}b^{2}_{pq;i}\\
&\quad-\sum_{i}\frac{m\sigma^{ii}_{k}}{P^{2}_{m}}(\sum_{j}\lambda^{m-1}_{j}b_{jj;i})^{2}+N\sum_{i}\sigma^{ii}_{k}b^{2}_{ii}+\sum_{i}\sigma^{ii}_{k}\left(-Nh^{2}_{i}+Nh^{2}-M-M\frac{h^{2}_{i}}{h^{2}} \right)-\hat{C},
\end{split}
\end{equation}
where the positive constants $C_{0}(W)$, $C_{1}(W)$ and $C_{2}(W)$ depend on the positive constant $W, \min_{\sn}h, ||h||_{C^{1}(\sn)},\min_{\sn}f $ and $||f||_{C^{2}(\sn)}$. Now, if we take $M<0$, $N>0$ and
\[
-M=N\max_{\sn}h^2.
\]
Then \eqref{mj} becomes
\begin{equation}
\begin{split}
\label{mj2}
0&\geq \frac{1}{P_{m}}\sum_{j}\lambda^{m-1}_{j}(-C_{0}(W)-C_{1}(W)b_{jj}-C_{2}(W)b^{2}_{jj}+W(\sigma_{k})^{2}_{j}-\sum_{p,q}\sigma^{pp,qq}_{k}b_{pp;j}b_{qq;j}+\sum_{p,q}\sigma^{pp,qq}_{k}b^{2}_{pq;j})\\
&\quad+\frac{1}{P_{m}}(m-1)\sum_{i}\sigma^{ii}_{k}\sum_{j}\lambda^{m-2}_{j}b^{2}_{jj;i}+\frac{1}{P_{m}}\sum_{i}\sigma^{ii}_{k}\sum_{p,q;p\neq q}\frac{\lambda^{m-1}_{p}-\lambda^{m-1}_{q}}{\lambda_{p}-\lambda_{q}}b^{2}_{pq;i}\\
&\quad-\sum_{i}\frac{m\sigma^{ii}_{k}}{P^{2}_{m}}(\sum_{j}\lambda^{m-1}_{j}b_{jj;i})^{2}+N\sum_{i}\sigma^{ii}_{k}b^{2}_{ii}-\hat{C}.
\end{split}
\end{equation}
Next we deal with the third-order derivatives. Denote
\[
A_{i}=\frac{\lambda^{m-1}_{i}}{P_{m}}(W(\sigma_{k})^{2}_{i}-\sum_{p,q}\sigma^{pp,qq}_{k}b_{pp;i}b_{qq;i}),\quad B_{i}=\frac{2}{P_{m}}\sum_{j}\lambda^{m-1}_{j}\sigma^{jj,ii}_{k}b^{2}_{jj;i},
\]

\[
C_{i}=\frac{m-1}{P_{m}}\sigma^{ii}_{k}\sum_{j}\lambda^{m-2}_{j}b^{2}_{jj;i},\quad D_{i}=\frac{2}{P_{m}}\sum_{j\neq i}\sigma^{jj}_{k}\frac{\lambda^{m-1}_{j}-\lambda^{m-1}_{i}}{\lambda_{j}-\lambda_{i}}b^{2}_{jj;i},
\]

\[
E_{i}=\frac{m\sigma^{ii}_{k}}{P^{2}_{m}}(\sum_{j}\lambda^{m-1}_{j}b_{jj;i})^{2}.
\]
We follow the approach in \cite{LRW16} to control the term involving the third-order derivatives, in the cases $i\neq 1$ and $i=1$.

We begin with the following lemma, which is a slightly modified version of \cite[Lemma 8]{LRW16}.
\begin{lem}\label{Lem1}
For any $i\neq 1$, we obtain
\[
A_{i}+B_{i}+C_{i}+D_{i}-\left( 1+\frac{1}{m} \right)E_{i}\geq 0,
\]
for sufficiently large $m$.
\end{lem}
\begin{proof}
By \cite[Lemma 2.2]{GRW15} (note that $\sigma^{pp,qq}_{1}=0$), when the constant $W$ is large, for $\alpha>0$, there holds
\begin{equation}\label{Ai}
W(\sigma_{k})^{2}_{i}-\sum_{p,q}\sigma^{pp,qq}_{k}b_{pp;i}b_{qq;i}\geq \sigma_{k}\left(1+\frac{\alpha}{2}  \right)\left[\frac{(\sigma_{1})_{i}}{\sigma_{1}}\right]^{2}\geq 0.
\end{equation}
Moreover, \eqref{Ai} implies
\begin{equation}\label{Ai2}
A_{i}\geq 0.
\end{equation}
On the other hand,
\begin{equation}
\begin{split}
\label{Bi}
&P^{2}_{m}\left[B_{i}+C_{i}+D_{i}-\left(1+\frac{1}{m}\right)E_{i}  \right]\\
&=\sum_{j\neq i}P_{m}\left(2\lambda^{m-1}_{j}\sigma^{jj,ii}_{k}+(m-1)\lambda^{m-2}_{j}\sigma^{ii}_{k}+2\sigma^{jj}_{k}\sum^{m-2}_{l=0}\lambda^{m-2-l}_{i}\lambda^{l}_{j}\right)b^{2}_{jj;i}\\
&\quad-(m+1)\sigma^{ii}_{k}\left(\sum_{j\neq i}\lambda^{2m-2}_{j}b^{2}_{jj;i}+\lambda^{2m-2}_{i}b^{2}_{ii;i}+\sum_{p,q;p\neq q}\lambda^{m-1}_{p}\lambda^{m-1}_{q}b_{pp;i}b_{qq;i}\right)\\
&\quad +(m-1)P_{m}\sigma^{ii}_{k}\lambda^{m-2}_{i}b^{2}_{ii;i},
\end{split}
\end{equation}
where we used
\[
\frac{\lambda^{m-1}_{j}-\lambda^{m-1}_{i}}{\lambda_{j}-\lambda_{i}}=\sum^{m-2}_{l=0}\lambda^{m-2-l}_{i}\lambda^{l}_{j}.
\]
Note that
\begin{equation}
\begin{split}
\label{Ci}
\lambda_{j}\sigma^{jj,ii}_{k}+\sigma^{jj}_{k}&=\sigma^{ii}_{k}-\sigma_{k-1}(\lambda|ij)+\lambda_{i}\sigma_{k-2}(\lambda|ij)+\sigma_{k-1} (\lambda|ij)\\
&=\lambda_{i}\sigma^{jj,ii}_{k}+\sigma^{ii}_{k}\geq \sigma^{ii}_{k}.
\end{split}
\end{equation}
For any index $j\neq i$, utilizing \eqref{Ci} into \eqref{Bi}, we find
\begin{equation}
\begin{split}
\label{Di}
&P_{m}\left(2\lambda^{m-1}_{j}\sigma^{jj,ii}_{k}+(m-1)\lambda^{m-2}_{j}\sigma^{ii}_{k}+2\sigma^{jj}_{k}\sum^{m-2}_{l=0}\lambda^{m-2-l}_{i}\lambda^{l}_{j}\right)b^{2}_{jj;i} -(m+1)\sigma^{ii}_{k}\lambda^{2m-2}_{j}b^{2}_{jj;i}\\
&\geq P_{m}(m+1)\sigma^{ii}_{k}\lambda^{m-2}_{j}b^{2}_{jj;i}-(m+1)\sigma^{ii}_{k}\lambda^{2m-2}_{j}b^{2}_{jj;i}+2P_{m}\sigma^{jj}_{k}\left( \sum^{m-3}_{l=0}\lambda^{m-2-l}_{i}\lambda^{l}_{j} \right)b^{2}_{jj;i}\\
&=(m+1)(P_{m}-\lambda^{m}_{j})\sigma^{ii}_{k}\lambda^{m-2}_{j}b^{2}_{jj;i}+2P_{m}\sigma^{jj}_{k}\left( \sum^{m-3}_{l=0}\lambda^{m-2-l}_{i}\lambda^{l}_{j} \right)b^{2}_{jj;i}.
\end{split}
\end{equation}
By the Cauchy-Schwarz inequality,
\begin{equation}
\begin{split}
\label{Ei}
&2\sum_{j\neq i}\sum_{p\neq i,j}\lambda^{m-2}_{j}\lambda^{m}_{p}b^{2}_{jj;i}\\
&=\sum_{p\neq i}\sum_{q\neq i,p}\lambda^{m-2}_{p}\lambda^{m}_{q}b^{2}_{pp;i}+\sum_{q\neq i}\sum_{p\neq i,q}\lambda^{m-2}_{q}\lambda^{m}_{p}b^{2}_{qq;i}\\
&\geq 2\sum_{p\neq q;p,q\neq i}\lambda^{m-1}_{p}\lambda^{m-1}_{q}b_{pp;i}b_{qq;i}.
\end{split}
\end{equation}
Applying \eqref{Bi}, \eqref{Di} and \eqref{Ei}, it yields
\begin{equation}
\begin{split}
\label{Fi}
&P^{2}_{m}\left[B_{i}+C_{i}+D_{i}-\left(1+\frac{1}{m}\right)E_{i}\right]\\
&\geq\sum_{j\neq i}(m+1)\lambda^{m}_{i}\lambda^{m-2}_{j}\sigma^{ii}_{k}b^{2}_{jj;i}+\sum_{j\neq i}\sum_{p\neq i,j}(m+1)\sigma^{ii}_{k}\lambda^{m-2}_{j}\lambda^{m}_{p}b^{2}_{jj;i}\\
&\quad+((m-1)(P_{m}-\lambda^{m}_{i})-2\lambda^{m}_{i})\lambda^{m-2}_{i}\sigma^{ii}_{k}b^{2}_{ii;i} \\
&\quad -\sum_{p,q;p\neq q}(m+1)\sigma^{ii}_{k}\lambda^{m-1}_{p}\lambda^{m-1}_{q}b_{pp;i}b_{qq;i}+2P_{m}\sum_{j\neq i}\sigma^{jj}_{k}\left(\sum^{m-3}_{l=0}\lambda^{m-2-l}_{i}\lambda^{l}_{j}\right)b^{2}_{jj;i}\\
&\geq \sum_{j\neq i}(m+1)\lambda^{m}_{i}\lambda^{m-2}_{j}\sigma^{ii}_{k}b^{2}_{jj;i}+((m-1)(P_{m}-\lambda^{m}_{i})-2\lambda^{m}_{i})\lambda^{m-2}_{i}\sigma^{ii}_{k}b^{2}_{ii;i}\\
&\quad-2(m+1)\sigma^{ii}_{k}\lambda^{m-1}_{i}b_{ii;i}\sum_{j\neq i}\lambda^{m-1}_{j}b_{jj;i}+2P_{m}\sum_{j\neq i}\sigma^{jj}_{k}\left(\sum^{m-3}_{l=0}\lambda^{m-2-l}_{i}\lambda^{l}_{j}\right)b^{2}_{jj;i}\\
&\geq \sum_{j\neq i}\left[(m+1)\lambda^{m}_{i}\lambda^{m-2}_{j}\sigma^{ii}_{k}+2\lambda^{m}_{1}\sigma^{jj}_{k}\sum^{m-3}_{l=0}\lambda^{m-2-l}_{i}\lambda^{l}_{j}\right]b^{2}_{jj;i}\\
&\quad-2(m+1)\sigma^{ii}_{k}\lambda^{m-1}_{i}b_{ii;i}\sum_{j\neq i}\lambda^{m-1}_{j}b_{jj;i}+((m-1)(P_{m}-\lambda^{m}_{i})-2\lambda^{m}_{i})\lambda^{m-2}_{i}\sigma^{ii}_{k}b^{2}_{ii;i}.
\end{split}
\end{equation}
Given \eqref{Fi}, we analyze two cases.

(Case I): When $\lambda_{j}\geq \lambda_{i}$, for $1\leq l \leq m-3$, we have
\begin{equation}
\begin{split}
\label{Gi}
2\lambda^{m}_{1}\sigma^{jj}_{k}\lambda^{m-2-l}_{i}\lambda^{l}_{j}&=2\lambda^{m}_{1}(\lambda_{i}\sigma^{ii,jj}_{k}+\sigma_{k-1}(\lambda|ij))\lambda^{m-2-l}_{i}\lambda^{l}_{j}\\
&\geq \lambda^{m}_{1}(\lambda_{i}\sigma^{ii,jj}_{k}+\sigma_{k-1}(\lambda|ij)\lambda^{m-2-l}_{i}\lambda^{l}_{j}\\
&\geq \lambda^{m}_{1}(\lambda_{j}\sigma^{ii,jj}_{k}+\sigma_{k-1}(\lambda|ij)\lambda^{m-1-l}_{i}\lambda^{l-1}_{j}\\
&=\lambda^{m}_{1}\lambda^{m-1-l}_{i}\lambda^{l-1}_{j}\sigma^{ii}_{k}.
\end{split}
\end{equation}

(Case II): When $\lambda_{j}<\lambda_{i}$, we obtain
\begin{equation}
\label{Hi}
2\lambda^{m}_{1}\sigma^{jj}_{k}\lambda^{m-2-l}_{i}\lambda^{l}_{j}\geq 2\lambda^{m}_{1}\lambda^{m-2-l}_{i}\lambda^{l}_{j}\sigma^{ii}_{k}.
\end{equation}

Combining \eqref{Gi} and \eqref{Hi}, when taking $m=\max\{10, k+11\}$, for $k\leq l\leq k+8$, we find
\begin{equation}
\label{Ii}
2\lambda^{m}_{1}\sigma^{jj}_{k}\lambda^{m-2-l}_{i}\lambda^{l}_{j}\geq \lambda^{m}_{i}\lambda^{m-2}_{j}\sigma^{ii}_{k}.
\end{equation}
Substituting \eqref{Ii} into \eqref{Fi}, for $m\geq 10$, we obtain
\begin{equation}
\begin{split}
\label{Ji}
&P^{2}_{m}\left[B_{i}+C_{i}+D_{i}-\left(1+\frac{1}{m}\right)E_{i}\right]\\
&\geq \sum_{j\neq i}(m+8)\lambda^{m}_{i}\lambda^{m-2}_{j}\sigma^{ii}_{k}b^{2}_{jj;i}-2(m+1)\sigma^{ii}_{k}\lambda^{m-1}_{i}b_{ii;i}\sum_{j\neq i}\lambda^{m-1}_{j}b_{jj;i}\\
&\quad+((m-1)(P_{m}-\lambda^{m}_{i})-2\lambda^{m}_{i})\lambda^{m-2}_{i}\sigma^{ii}_{k}b^{2}_{ii;i}.
\end{split}
\end{equation}
Following the same steps in  \cite[Eq. (3.17)]{LRW16}, given $m\geq 10$, for any $i\neq 1$, \eqref{Ji} further turns into
\begin{equation}
\begin{split}
\label{Li}
&P^{2}_{m}\left[B_{i}+C_{i}+D_{i}-\left(1+\frac{1}{m}\right)E_{i}\right]\\
&\geq (m+8)\lambda^{m}_{i}\lambda^{m-2}_{1}\sigma^{ii}_{k}b^{2}_{11;i}-2(m+1)\sigma^{ii}_{k}\lambda^{m-1}_{i}b_{ii;i}\lambda^{m-1}_{1}b_{11;i}\\
&\quad+((m-1)\lambda^{m}_{1}-2\lambda^{m}_{i})\lambda^{m-2}_{i}\sigma^{ii}_{k}b^{2}_{ii;i}\\
&\geq (m+8)\lambda^{m}_{i}\lambda^{m-2}_{1}\sigma^{ii}_{k}b^{2}_{11;i}-2(m+1)\sigma^{ii}_{k}\lambda^{m-1}_{i}b_{ii;i}\lambda^{m-1}_{1}b_{11;i}\\
&\quad+(m-3)\lambda^{m}_{1}\lambda^{m-2}_{i}\sigma^{ii}_{k}b^{2}_{ii;i}\\
&\geq 0.
\end{split}
\end{equation}
Hence, from \eqref{Ai2} and \eqref{Li}, the proof is complete.
\end{proof}

 For the case $i=1$, we can also verbatim follow the proof of \cite[Lemma 9]{LRW16} with minor changes to prove the following lemma. For completeness, we sketch the argument below.
\begin{lem}\label{Lem2}
For $l=1,\ldots,k-1$, if there exist some positive constants $\delta\leq 1$ such that $\lambda_{l}/\lambda_{1}\geq \delta$.
Then there exist two sufficiently small positive constants $\eta,\delta^{'}$ depending on $\delta$, such that, if $\lambda_{l+1}/\lambda_{1}\leq \delta^{'}$, we obtain
\[
A_{1}+B_{1}+C_{1}+D_{1}-\left(1+\frac{\eta}{m} \right)E_{1}\geq 0,
\]
for sufficiently large $m$.
\end{lem}
\begin{proof}
 Similar to \eqref{Fi}, one hand, we have
\begin{equation}
\begin{split}
\label{V1}
&P^{2}_{m}\left[B_{1}+C_{1}+D_{1}-\left( 1+\frac{\eta}{m}\right)E_{1}\right]\\
&\geq \sum_{j\neq 1}((1-\eta)P_{m}+(m+\eta)\lambda^{m}_{1})\lambda^{m-2}_{j}\sigma^{11}_{k}b^{2}_{jj;1}\\
&\quad-2(m+\eta)\sigma^{11}_{k}\lambda^{m-1}_{1}b_{11;1}\sum_{j\neq 1}\lambda^{m-1}_{j}b_{jj;1}+2P_{m}\sum_{j\neq 1}\sigma^{jj}_{k}\left(\sum^{m-3}_{l=0}\lambda^{m-2-l}_{1}\lambda^{l}_{j}\right)b^{2}_{jj;1}\\
&\quad+((m-1)(P_{m}-\lambda^{m}_{1})-(1+\eta)\lambda^{m}_{1})\lambda^{m-2}_{1}\sigma^{11}_{k}b^{2}_{11;1}.
\end{split}
\end{equation}
 For  large $m$ with $m\geq 5$,  according to  \cite[Eq. (3.19)]{LRW16}, \eqref{V1} further becomes
\begin{equation}
\begin{split}
\label{V2}
&P^{2}_{m}\left[B_{1}+C_{1}+D_{1}-\left( 1+\frac{\eta}{m}\right)E_{1}\right]\\
&\geq -(1+\eta)\lambda^{2m-2}_{1}\sigma^{11}_{k}b^{2}_{11;1}+2P_{m}\lambda^{m-2}_{1}\sum_{j \neq1}\sigma^{jj}_{k}b^{2}_{jj;1}.
\end{split}
\end{equation}
On the other hand, using again \cite[Lemma 2.2]{GRW15}, for large $m$, given $\alpha:=1/(k-l)$, there holds
\begin{equation}
\begin{split}
\label{V3}
A_{1}&\geq \frac{\lambda^{m-1}_{1}}{P_{m}}\left[ \sigma_{k}\left(1+\frac{\alpha}{2}\right)\frac{(\sigma_{l})^{2}_{1}}{\sigma^{2}_{l}}-\frac{\sigma_{k}}{\sigma_{l}} \sigma^{pp,qq}_{l}b_{pp;1}b_{qq;1} \right]\\
&= \frac{\lambda^{m-1}_{1}\sigma_{k}}{P_{m}\sigma^{2}_{l}}\left[\left(1+\frac{\alpha}{2}\right)\sum_{s}(\sigma^{ss}_{l}b_{ss;1})^{2}+\frac{\alpha}{2} \sum_{s\neq r}\sigma^{ss}_{l}\sigma^{rr}_{l}b_{ss;1}b_{rr;1}   \right.\\
&\left. \quad \quad \quad \quad \quad \quad +\sum_{s\neq r}(\sigma^{ss}_{l}\sigma^{rr}_{l}-\sigma_{l}\sigma^{ss,rr}_{l})b_{ss;1}b_{rr;1}\right].
\end{split}
\end{equation}
Given \eqref{V3}, we analyze two cases.

Case (I): When $l=1$, along the same lines as \cite[Eq. (3.22)]{LRW16}, for some positive constant $C_{\alpha}$ depending on $\alpha$,  we find
\begin{equation}
\begin{split}
\label{V4}
P^{2}_{m}A_{1}\geq (1+\eta)P_{m}\lambda^{m-2}_{1}\sigma^{11}_{k}b^{2}_{11;1}-\frac{C_{\alpha}\sigma_{k}P_{m}\lambda^{m-1}_{1}}{\sigma^{2}_{1}}\sum_{s\neq 1}b^{2}_{ss;1},
\end{split}
\end{equation}
where
\[
1+\frac{\alpha}{4}\geq (1+\eta)(1+(n-1)\delta^{'})^{2}.
\]

Case (II): When $l\geq 2$, as proved in \cite{LRW16},  one sees
\begin{equation}\label{V5}
\sum_{s\neq r}(\sigma^{ss}_{l}\sigma^{rr}_{l}-\sigma_{l}\sigma^{ss,rr}_{l})b_{ss;1}b_{rr;1}\geq -2\epsilon \sum_{s\leq l}(\sigma^{ss}_{l}b_{ss;1})^{2}-C_{\epsilon}\sum_{s> l}(\sigma^{ss}_{l}b_{ss;1})^{2},
\end{equation}
where $\epsilon$ is some positive  constant and $C_{\epsilon}$ is a positive constant depending on $\epsilon$. Applying \eqref{V5} into \eqref{V3}, similar to \cite[Eq. (3.33)]{LRW16}, for positive constants $\tilde{C}$ and $\tilde{C}_{\epsilon}$ depending on $n,k$ and $n,k,\epsilon$ respectively, we obtain
\begin{equation}
\begin{split}
\label{V6}
P^{2}_{m}A_{1}
&\geq \frac{P_{m}\lambda^{m}_{1}\sigma^{11}_{k}}{\sigma^{2}_{l}}(1-2\epsilon)\sum_{s\leq l}(\sigma^{ss}_{l}b_{ss;1})^{2}-\frac{P_{m}\lambda^{m-1}_{1}\sigma_{k}C_{\epsilon}}{\sigma^{2}_{l}}\sum_{s> l}(\sigma^{ss}_{l}b_{ss;1})^{2}\\
&\geq \lambda^{2m-2}_{1}\sigma^{11}_{k}(1-2\epsilon)(1+\delta^{m})\left( 1-\frac{\tilde{C}\lambda_{l+1}}{\delta\lambda_{1}} \right)^{2}\sum_{s\leq l}b^{2}_{ss;1}-\frac{P_{m}\lambda^{m-3}_{1}\sigma_{k}\tilde{C}_{\epsilon}}{\delta^{2}}\sum_{s>l}b^{2}_{ss;1}\\
&\geq (1+\eta)\lambda^{2m-2}_{1}\sigma^{11}_{k}\sum_{s\leq l}b^{2}_{ss;1}-\frac{P_{m}\lambda^{m-3}_{1}\sigma_{k}\tilde{C}_{\epsilon}}{\delta^{2}}\sum_{s>l}b^{2}_{ss;1},
\end{split}
\end{equation}
where $\delta^{'},\eta$ and $\epsilon$ satisfy
\[
\delta^{'}\tilde{C}\leq 2\epsilon \delta, \quad (1-2\epsilon)^{3}(1+\delta^{m})\geq 1+\eta.
\]

Combining \eqref{V2} with \eqref{V4} and \eqref{V6}, we have
\begin{equation}
\begin{split}
\label{V7}
&P^{2}_{m}\left[A_{1}+B_{1}+C_{1}+D_{1}-\left( 1+\frac{\eta}{m}\right)E_{1}\right]\\
&\geq 2P_{m}\lambda^{m-2}_{1}\sum_{j \neq1}\sigma^{jj}_{k}b^{2}_{jj;1}-\frac{P_{m}\sigma_{k}\lambda^{m-3}_{1}\tilde{C}_{\epsilon}}{\delta^{2}}\sum_{j>l}b^{2}_{jj;1}.
\end{split}
\end{equation}
Since for $j>l$, there holds
\begin{equation}\label{V8}
\lambda_{1}\sigma_{k-1}(\lambda|j)\geq \frac{\sigma_{k}}{\tilde{C}\delta^{'}},
\end{equation}
which comes from
\[
\frac{\sigma_{k}}{\lambda_{1}}\leq \frac{\delta^{'}\sigma_{k}}{\lambda_{j}}\leq \frac{\tilde{C}\delta^{'}\lambda_{1}\ldots \lambda_{k}}{\lambda_{j}}\leq \tilde{C}\delta^{'}\sigma^{jj}_{k}
\]
for $l<j\leq k$ and
\[
\frac{\sigma_{k}}{\lambda_{1}}\leq \frac{\delta^{'}\sigma_{k}}{\lambda_{k}}\leq \tilde{C}\delta^{'}\lambda_{1}\ldots \lambda_{k-1}\leq \tilde{C}\delta^{'}\sigma^{jj}_{k}
\]
for $j>k$.

Now using \eqref{V8} into \eqref{V7}, choosing $\delta^{'}$ small enough that satisfies
\[
\delta^{'}< \frac{\delta^{2}}{\tilde{C}\tilde{C}_{\epsilon}},
\]
thus \eqref{V7} is nonnegative. Hence the proof is complete.

\end{proof}

Similar to  \cite[Corollary 10]{LRW16},  we have the following result.
\begin{lem}\label{Lem3}
There exist two finite sequences of positive numbers $\{\delta_{j}\}^{k}_{j=1}$ and $\{\xi_{j}\}^{k}_{j=1}$ such that, if the following inequality holds for some index $1\leq s\leq k-1$,
\[
\frac{\lambda_{s}}{\lambda_{1}}\geq \delta_{s}, \ {\rm and} \ \frac{\lambda_{s+1}}{\lambda_{1}}\leq \delta_{s+1},
\]
then for large $W$, there holds
\begin{equation}\label{Oi}
A_{1}+B_{1}+C_{1}+D_{1}-\left(1+\frac{\xi_{s}}{m} \right)E_{1}\geq 0.
\end{equation}
\end{lem}
\begin{proof}
Using induction to find the sequences $\{\delta_{j}\}^{k}_{j=1}$ and $\{\xi_{j}\}^{k}_{j=1}$. We first set $\delta_{1}=1/2$, this choice satisfies $\lambda_{1}/\lambda_{1}=1>\delta_{1}$. That the claim holds for $j=1$ via Lemma \ref{Lem2}.
Assume that $\delta_{s}$ has been defined for $1\leq s\leq k-1$. To determine $\delta_{s+1}$, we apply Lemma \ref{Lem2} with the parameters $l=s$ and $\delta=\delta_{s}$. Then there is some $\delta^{'}_{s+1}$, such that if $\lambda_{s+1}\leq \delta^{'}_{s+1}\lambda_{1}$, we have \eqref{Oi} with some $\xi_{s}$. Pick $\delta_{s+1}=\min\{\delta_{1},\delta^{'}_{s+1}\}$, thus \eqref{Oi} is satisfied if $\lambda_{s+1}\leq \lambda_{1}\delta_{s+1}$. So $\delta_{s+1}$ and $\xi_{s}$ are determined.
\end{proof}

\begin{proof}[Proof of Theorem \ref{hc2}.]

To prove it, we analyze two cases.

Case (I): There exists some index $1\leq s\leq k-1$ and positive $\{\delta_{j}\}^{k}_{j=1}$ such that
\[
\lambda_{s}\geq \delta_{s}\lambda_{1},\ {\rm and} \ \lambda_{s+1}\leq \delta_{s+1}\lambda_{1}.
\]
Combining Lemma \ref{Lem1} and Lemma \ref{Lem3}, we get
\begin{equation}\label{Pi}
\sum_{i}(A_{i}+B_{i}+C_{i}+D_{i})-E_{1}-\left(1+\frac{1}{m} \right)\sum^{n}_{i=2}E_{i}\geq 0.
\end{equation}
By the definition of $A_{i}$, $B_{i}$, $C_{i}$, $D_{i}$, $E_{i}$, and employing \eqref{Pi} into \eqref{mj2}, we have
\begin{equation}
\begin{split}
\label{Qi}
0&\geq \frac{1}{P_{m}}\sum_{j}\lambda^{m-1}_{j}(-C_{0}(W)-C_{1}(W)b_{jj}-C_{2}(W)b^{2}_{jj})
\\
&\quad+\sum^{n}_{i=2}\frac{\sigma^{ii}_{k}}{P^{2}_{m}}(\sum_{j}\lambda^{m-1}_{j}b_{jj;i})^{2}+N\sum_{i}\sigma^{ii}_{k}b^{2}_{ii}-\hat{C}\\
&\geq -\frac{\tilde{C}_{0}(W)}{\lambda_{1}}-\tilde{C}_{1}(W)-\hat{C}-\tilde{C}_{2}(W)\lambda_{1}+N\sigma^{11}_{k}b^{2}_{11}.
\end{split}
\end{equation}
By the Newton-MacLaurin inequality, we have
\begin{equation}\label{Mac}
\left[\frac{\sigma_{k-1}(\lambda|1)}{\binom{n-1}{k-1}}\right]^\frac{1}{k-1}\geq \left[\frac{\sigma_{k}(\lambda|1)}{\binom{n-1}{k}}\right]^\frac{1}{k},
\end{equation}
using \eqref{Mac}, for some positive constants $C_{n,k}$, depending only on $n,k$, we get
\begin{equation}
\begin{split}
\label{Fw11}
\sigma_{k}(\lambda|1)\leq C_{n,k} \sigma_{k-1}(\lambda|1)^\frac{k}{k-1}\leq C_{n,k}\lambda_{1}\sigma_{k-1}(\lambda|1).
\end{split}
\end{equation}
Substituting \eqref{Fw11} into $\sigma_{k}(\lambda)=\sigma_{k}(\lambda|1)+\lambda_{1}\sigma_{k-1}(\lambda|1)$, we have
\begin{equation}\label{1cq}
\lambda_{1}\sigma_{k-1}(\lambda|1)\geq C_{n,k} \sigma_{k}(\lambda).
\end{equation}
So by \eqref{1cq}, we have
\begin{equation*}
\begin{split}
\label{maxw11}
\frac{\sigma_{k}^{11}b^{2}_{11}}{\sigma_{k}}&=\frac{\sigma_{k-1}(\lambda|1)\lambda^{2}_{1}}{\sigma_{k}}\geq \frac{C_{n,k}\sigma_{k}\lambda_{1}}{\sigma_{k}}=C_{n,k}b_{11}.
\end{split}
\end{equation*}
It follows that
\begin{equation}\label{Ti}
\sigma^{11}_{k}b^{2}_{11}\geq C^{*}b_{11},
\end{equation}
where the positive constant $C^{*}$ depending on $n,k,\min_{\sn}h,||h||_{C^{1}(\sn)}, \min_{\sn}f$ and $||f||_{C^{2}(\sn)}$.
Now applying \eqref{Ti} into \eqref{Qi}, choosing  $N= \frac{\tilde{C}_{2}(W)+1}{C^{*}}$, we get
\begin{equation}
\begin{split}
\label{Ri}
0&\geq -\frac{\tilde{C}_{0}(W)}{\lambda_{1}}-\tilde{C}_{1}(W)-\hat{C}-\tilde{C}_{2}(W)\lambda_{1}+C^{*}N\lambda_{1}\\
&\geq -\frac{\tilde{C}_{0}(W)}{\lambda_{1}}-\tilde{C}_{1}(W)-\hat{C}+\lambda_{1},
\end{split}
\end{equation}
then when provided $\lambda_{1}\gg1$, from \eqref{Ri}, we get
\[
\lambda_{1}\leq C
\]
for a positive constant $C$.

Case (II): $\lambda_{k}\geq \delta_{k}\lambda_{1}$ with $\delta_{k}>0$. Since $\lambda_{1}\geq \lambda_{2}\geq \ldots \geq \lambda_{k}\geq \delta_{k}\lambda_{1}$ and $\lambda_{i}\geq 0$ for all $i$, then
\[
f=\sigma_{k}>\lambda_{1}\ldots \lambda_{k}\geq \delta^{k-1}_{k}\lambda^{k}_{1},
\]
which also illustrates $\lambda_{1}\leq C$ for a positive constant $C$. Hence the proof of Theorem \ref{hc2} is complete.

\end{proof}

In the case of $1<p<q\leq k+1$, it is difficult to determine the kernel of the related linearized operator of \eqref{PQ}. So, we employ the degree theory method to prove Theorem \ref{Thm1}.

\begin{proof} [Proof of Theorem \ref{Thm1}.]

For $0<\alpha<1$  and $0\leq t\leq 1$,  denote by $T_{t}(\cdot): C^{4,\alpha}(\sn)\rightarrow C^{2,\alpha}(\sn)$ a nonlinear differential operator as
\begin{equation*}
T_{t}(h)=\sigma_{k}(\nabla^{2}h+hI)-h^{p-1}(|\nabla h|^{2}+h^{2})^{\frac{k+1-q}{2}}f_{t},
\end{equation*}
where
\[
f_{t}=\left[(1-t)\binom{n}{k}^{-\frac{1}{k+p-1}}+tf^{-\frac{1}{k+p-1}}\right]^{-(k+p-1)}.
\]
Based on the construction of $f_{t}$ and the assumption on $f$ specified in \eqref{f1}, we can see that $f_{t}$ satisfies the condition in Theorem \ref{Sph}; this illustrates that the solution $h(x,t)$ of $T_{t}(h)=0$ is strictly convex. Moreover, by Lemma \ref{Bou} and Theorem \ref{hc2}, we have the uniform $C^{2}$ estimate on $h(x,t)$.  With the help of the classical Evans-Krylov theorem and the Schauder regularity theory, for any fixed $t\in [0,1]$ and $\ell\geq 2$, there exists a constant $\bar{C}>0$, depending only on $n,k,\gamma,\min_{\sn}h,||h||_{C^{1}(\sn)}$ and $||f||_{C^{2}(\sn)}$ such that the following uniform higher order estimate holds:
\begin{equation}\label{hh}
||h||_{C^{\ell+1,\alpha}}\leq \bar{C}.
\end{equation}

Now, let $R>0$ be fixed, define $\mathcal{O}\subset C^{4,\alpha}(\sn)$ as
\[
\mathcal{O}=\{h\in C^{4,\alpha}(\sn): \ h(x)=h(-x) \ \forall x\in \sn, \ \nabla^{2}h+hI>0, \  ||h||_{C^{4,\alpha}(\sn)}< R  \}.
\]
From \eqref{hh} and Theorem \ref{Sph}, we know that if $R$ is sufficiently large, $T_{t}(h)=0$ has no solution on $\partial \mathcal{O}$. Therefore, the degree of $T_{t}$ is well defined (see e.g. \cite[Sec. 2]{LY89}). Since degree is homotopy invariant,
\[
\deg(T_{0}, \mathcal{O}, 0)=\deg(T_{1}, \mathcal{O}, 0).
\]
One hand, at $t=0$,  the following Theorem \ref{UC} implies that $h=1$ is the unique solution of $\eqref{PQ}$ when $f=\binom{n}{k}$ (see also \cite{DL23,LW24}).  On the other hand, since $T$ is symmetric, we can show that for any $\phi\in C^{2}(\sn)$, the linearized operator of $T_{0}$  at $h=1$ is
\[
L_{0}\phi=\binom{n-1}{k-1}(\Delta \phi +n\phi)+(q-p-k)\binom{n}{k}\phi.
\]
 If $L_{0}\phi=0$, that is, $\Delta \phi+\frac{n(q-p)}{k}\phi=0$. For $1<p<q\leq k+1$, then $\frac{n(q-p)}{k}<n$.
Since $\frac{n(q-p)}{k}$ is not an eigenvalue of $(-\Delta)$ on $\sn$, this illustrates that the linearized operator $L_{0}$ is invertible. With the aid of this fact, we calculate the degree by applying the formula
\[
{\rm deg}(T_{0},\mathcal{O},0)=\sum_{\mu_{j}>0}(-1)^{\varpi_{j}}.
\]
Here $\mu_{j}>0$ represents the eigenvalues of the linearized operator of $T_{0}$ and $\varpi_{j}$ denotes their corresponding multiplicities. Since the eigenvalues of the Beltrami-Laplace operator $\Delta$ on $\sn$ are strictly less than $-n$, except the first two eigenvalues $0$ and $-n$, this implies that for $1<p<q\leq k+1$, the operator $L_{0}$ has only one positive eigenvalue with multiplicity 1, given by $\mu=\binom{n-1}{k-1}\frac{n(q-p)}{k}$. Thus we have
\[
\deg(T_{0}, \mathcal{O}, 0)=\deg(T_{1}, \mathcal{O}, 0)=-1\neq 0.
\]
Therefore, we get the existence of even solutions to \eqref{PQ}, and the regularity of solutions is directly derived from \eqref{hh}.
\end{proof}

\section{Uniqueness of solutions to \eqref{PQ} with $f=1$}
\label{Sec4}
We consider the isotropic case of \eqref{PQ} with $f=1$:
\begin{equation}\label{MP}
h^{1-p}\sigma_{k}(\nabla^{2} h+ hI)(|\nabla h|^{2}+h^{2})^{\frac{\beta}{2}}=1,\quad {\rm on}\ \sn,
\end{equation}
where $1\leq k < n$ and $\beta:=(q-k-1)\in \rn$.

It is worth noting that the uniqueness of solutions to \eqref{MP} plays an important role in deriving the existence of solutions to \eqref{PQ} with $1 \leq k < n$ by the degree theory method. Several uniqueness results are known: for $p > 1$ and $q \leq k + 1$, Ding-Li \cite{DL23} and Li-Wan \cite{LW24} independently established the uniqueness results without requiring symmetric conditions, while Ivaki-Milman \cite{IM23} obtained the uniqueness under symmetric conditions when $p>1-k$ and $q\leq k+1$. Additionally, for $p > 1$ and $q \leq p$, uniqueness was also investigated in \cite{DL23}. Note that when $k = n$ in \eqref{MP}, the uniqueness and non-uniqueness results of solutions to the isotropic $L_p$ dual Minkowski problem have been widely investigated under both symmetric and non-symmetric conditions. Relevant works can be found in \cite{CHZ19, CCL21, HI242, LW22, LW24}.

Our purpose is to give a new uniqueness result for \eqref{MP} by applying the local Aleksandrov-Fenchel inequality. The following lemma represents the spectral formulation of the local Aleksandrov-Fenchel inequality. This formulation traces back to Hilbert's work (see, e.g. \cite{AN97, AB20}) and has also appeared in \cite{IM23, KM22}.

\begin{lem}\label{SF}\cite{AN97, AB20}
Let $f\in C^{2}(\sn)$ with $\int_{\sn}fh\sigma_{k}d\sigma=0$. Then there holds
\begin{equation*}
k\int_{\sn}f^{2}h\sigma_{k}d\sigma\leq \int_{\sn}h^{2}\sigma^{ij}_{k}\nabla_{i}f\nabla_{j}fd\sigma,
\end{equation*}
where $\sigma^{ij}_{k}:=\frac{\partial \sigma_{k}}{\partial b_{ij}}$ with $b_{ij}=h_{ij}+h\delta_{ij}$. Equality holds if and only if for some vector $v\in \rnnn$ we have
\[
f(x)=\langle \frac{x}{h(x)},v\rangle,  \quad \forall x\in \sn.
\]
\end{lem}

Using Lemma \ref{SF}, we have the following result.
\begin{lem}\label{IN}
Let $X=Dh:\sn\rightarrow \partial K$ and $\alpha\in \rn$. Then we obtain
\begin{equation}
\begin{split}
\label{aa3}
&k\int_{\sn}|X|^{2\alpha+2}dV_{k}+(\alpha^{2}+2\alpha)\int_{\sn}h^{2}|X|^{2\alpha-2}\sum_{i}\lambda_{i}\sigma^{ii}_{k+1}h^{2}_{i}d\sigma\\
&\leq k\frac{\Big|\int_{\sn}|X|^{\alpha}X dV_{k}\Big|^{2}}{\int_{\sn}dV_{k}}+\int_{\sn}|X|^{2\alpha}h(\Delta h+nh)d V_{k}-\int_{\sn}h|X|^{2\alpha}(k+1)\frac{\sigma_{k+1}}{\sigma_{k}}dV_{k}\\
&\quad + (\alpha^{2}+2\alpha)\int_{\sn}|X|^{2\alpha-1}h\langle \nabla h, \nabla |X|\rangle d V_{k},
\end{split}
\end{equation}
where $dV_{k}=h\sigma_{k}d\sigma$.
\end{lem}

\begin{proof}
Let $\{E_{l}\}^{n+1}_{l=1}$ be an orthonormal basis of $\rnnn$. Assume that $\{e_{i}\}^{n}_{i=1}$ is a local orthonormal frame of $\sn$ such that $(h_{ij}+h\delta_{ij})(x_{0})=\lambda_{i}(x_{0})\delta_{ij}$. Inspired by \cite[Lemma 3.2]{IM23}, for $l=1,\ldots,n+1$, we define
the functions $f_{l}:\sn\rightarrow \rn$ as
\begin{equation*}
f_{l}(x)=|X(x)|^{\alpha}\langle X(x),E_{l}\rangle-\frac{\int_{\sn}|X(x)|^{\alpha}\langle X(x),E_{l}\rangle dV_{k}}{\int_{\sn}dV_{k}}.
\end{equation*}
Due to $\int_{\sn}f_{l}dV_{k}=0$ for $1\leq l \leq n+1$, applying Lemma \ref{SF} to $f_{l}$ and summing over $l$, we obtain
\begin{equation}\label{nc}
k\sum_{l}\int_{\sn}f^{2}_{l}dV_{k}=k\left[ \int_{\sn}|X|^{2\alpha+2}dV_{k}-\frac{\Big|\int_{\sn}|X|^{\alpha}X dV_{k}\Big|^{2}}{\int_{\sn}dV_{k}} \right]\leq \sum_{l,i,j}\int_{\sn}h^{2}\sigma^{ij}_{k}\nabla_{i}f_{l}\nabla_{j}f_{l}d\sigma.
\end{equation}
Note that $\nabla_{i}X=\sum_{j}(h_{ij}+h\delta_{ij})e_{j}=\lambda_{i}e_{i}$, $\langle e_{i}, X\rangle=h_{i}$ and $\sum_{i}\lambda_{i}\langle e_{i},X\rangle^{2}=|X|\langle \nabla h, \nabla |X|\rangle$ at $x_{0}$. Hence, using $\sigma^{ii}_{k+1}=\sigma_{k}-\lambda_{i}\sigma^{ii}_{k}$ and $\sum_{i}\frac{\partial \sigma_{k}}{\partial \lambda_{i}}\lambda^{2}_{i}=\sigma_{1}\sigma_{k}-(k+1)\sigma_{k+1}$, we obtain
\begin{equation}
\begin{split}
\label{acb}
\sum_{l,i,j}\sigma^{ij}_{k}\nabla_{i}f_{l}\nabla_{j}f_{l}&=\sum_{l,i}\frac{\partial \sigma_{k}}{\partial \lambda_{i}}((\nabla_{i}|X|^{\alpha})\langle X, E_{l}\rangle+|X|^{\alpha}\langle \nabla_{i}X,E_{l}\rangle)^{2}\\
&=\sum_{l,i}\frac{\partial \sigma_{k}}{\partial \lambda_{i}}(\alpha|X|^{\alpha-2}\langle \lambda_{i}e_{i},X\rangle\langle X,E_{l}\rangle+|X|^{\alpha}\langle \lambda_{i}e_{i},E_{l}\rangle)^{2}\\
&=\sum_{i}\frac{\partial \sigma_{k}}{\partial \lambda_{i}}\lambda^{2}_{i}(|X|^{2\alpha}+(\alpha^{2}+2 \alpha)|X|^{2\alpha-2}\langle e_{i},X\rangle^{2})\\
&=|X|^{2\alpha}(\sigma_{1}\sigma_{k}-(k+1)\sigma_{k+1})+(\alpha^{2}+2\alpha)|X|^{2\alpha-1}\langle \nabla h, \nabla |X|\rangle\sigma_{k}\\
&\quad -(\alpha^{2}+2\alpha)|X|^{2\alpha-2}\sum_{i}\lambda_{i}\sigma^{ii}_{k+1}h^{2}_{i}.
\end{split}
\end{equation}
Now the claim follows from substituting \eqref{acb} into \eqref{nc}:
\begin{equation*}
\begin{split}
&k\int_{\sn}|X|^{2\alpha+2}dV_{k}+(\alpha^{2}+2\alpha)\int_{\sn}h^{2}|X|^{2\alpha-2}\sum_{i}\lambda_{i}\sigma^{ii}_{k+1}h^{2}_{i}d\sigma\\
&\leq k\frac{\Big|\int_{\sn}|X|^{\alpha}X dV_{k}\Big|^{2}}{\int_{\sn}dV_{k}}+\int_{\sn}h^{2}|X|^{2\alpha}(\Delta h+nh)\sigma_{k}d\sigma-\int_{\sn}h^{2}|X|^{2\alpha}(k+1)\sigma_{k+1}d\sigma\\
&\quad + (\alpha^{2}+2\alpha)\int_{\sn}h^{2}|X|^{2\alpha-1}\langle \nabla h, \nabla |X|\rangle \sigma_{k}d \sigma\\
&= k\frac{\Big|\int_{\sn}|X|^{\alpha}X dV_{k}\Big|^{2}}{\int_{\sn}dV_{k}}+\int_{\sn}|X|^{2\alpha}h(\Delta h+nh)d V_{k}-\int_{\sn}h|X|^{2\alpha}(k+1)\frac{\sigma_{k+1}}{\sigma_{k}}dV_{k}\\
&\quad + (\alpha^{2}+2\alpha)\int_{\sn}|X|^{2\alpha-1}h\langle \nabla h, \nabla |X|\rangle d V_{k}.
\end{split}
\end{equation*}
The proof is complete.
\end{proof}

We are in a position to prove a new uniqueness result to \eqref{MP} by employing Lemma \ref{IN}.
\begin{theo}\label{UC}
Suppose $n\geq 2$.  Let $1\leq k <n$. Suppose
\begin{equation*}
1-k\leq p, \quad  q \leq k+1+2k-2n-2+2\sqrt{(n-k+1)^{2}+\frac{k+p-1}{n+2}},
\end{equation*}
and at least one of these inequalities is strict. Then the smooth, origin-symmetric and strictly convex solution $\partial K$ to Eq. \eqref{MP} is an origin-centred sphere.
\end{theo}

\begin{proof}
Let $d V_{k}=h^{p}|X|^{-\beta}d\sigma$. Since $K$ is origin-symmetric, $\int_{\sn}|X|^{\alpha}X dV_{k}=0$ and
 by integration by parts,
\begin{equation}
\begin{split}
\label{bb}
&\int_{\sn}|X|^{2\alpha}h \Delta h dV_{k}=\int_{\sn} |X|^{2\alpha-\beta}h^{p+1}\Delta h d\sigma\\
&=-(p+1)\int_{\sn}|X|^{2\alpha-\beta}h^{p}|\nabla h|^{2}d\sigma-(2\alpha-\beta)\int_{\sn}|X|^{2\alpha-\beta-1}h^{p+1}\langle \nabla h, \nabla |X|\rangle d\sigma\\
&=-(p+1)\int_{\sn}|X|^{2\alpha}|\nabla h|^{2}dV_{k}-(2\alpha-\beta)\int_{\sn}|X|^{2\alpha-1}h\langle \nabla h, \nabla |X|\rangle d V_{k}.
\end{split}
\end{equation}
Substituting \eqref{bb} into \eqref{aa3}, we obtain
\begin{equation}
\begin{split}
\label{dd}
&(k+p+1)\int_{\sn}|X|^{2\alpha}|\nabla h|^{2}dV_{k}+(\alpha^{2}+2\alpha)\int_{\sn}h^{2}|X|^{2\alpha-2}\sum_{i}\lambda_{i}\sigma^{ii}_{k+1}h^{2}_{i}d\sigma\\
&\leq (\alpha^{2}+\beta)\int_{\sn}|X|^{2\alpha-1}h\langle \nabla h, \nabla |X|\rangle dV_{k}-\int_{\sn}|X|^{2\alpha}(k+1)h^{2}\sigma_{k+1}d\sigma\\
&\quad +\int_{\sn}(n-k)|X|^{2\alpha}h^{2}dV_{k}.
\end{split}
\end{equation}
Using $(k+1)\sigma_{k+1}=\sum_{i,j}\sigma^{ij}_{k+1}(h_{ij}+h\delta_{ij})$ and integration by parts again, we have
\begin{equation}
\begin{split}
\label{cc}
&\int_{\sn}|X|^{2\alpha}(k+1)h^{2}\sigma_{k+1}d\sigma=\int_{\sn}|X|^{2\alpha}h^{2}\sum_{i,j}\sigma^{ij}_{k+1}(h_{ij}+h\delta_{ij})d\sigma\\
&=\int_{\sn}|X|^{2\alpha}h^{2}\sum_{i,j}\sigma^{ij}_{k+1}h_{ij}d\sigma+\int_{\sn}h^{3}|X|^{2\alpha}(n-k)\sigma_{k}d\sigma\\
&=-2\alpha\int_{\sn}|X|^{2\alpha-2}h^{2}\sum_{i}\lambda_{i}\sigma^{ii}_{k+1}h^{2}_{i}d\sigma-2\int_{\sn}|X|^{2\alpha}h\sum_{i}\sigma^{ii}_{k+1}h^{2}_{i}d\sigma+(n-k)\int_{\sn}|X|^{2\alpha}h^{2} dV_{k}\\
&=-2\alpha\int_{\sn}|X|^{2\alpha-2}h^{2}\sum_{i}\lambda_{i}\sigma^{ii}_{k+1}h^{2}_{i}d\sigma-2\int_{\sn}|X|^{2\alpha}h\sigma_{k}|\nabla h|^{2} d\sigma+2\int_{\sn}|X|^{2\alpha}h\sum_{i}\lambda_{i}\sigma^{ii}_{k}h^{2}_{i}d\sigma\\
&\quad +(n-k)\int_{\sn}|X|^{2\alpha}h^{2} dV_{k}.
\end{split}
\end{equation}
Substituting  \eqref{cc} into \eqref{dd}, we find
\begin{equation}
\begin{split}
\label{ee}
&(k+p-1)\int_{\sn}|X|^{2\alpha}|\nabla h|^{2}dV_{k}+2\int_{\sn}|X|^{2\alpha}h\sum_{i}\lambda_{i}\sigma^{ii}_{k}h^{2}_{i}d\sigma\\
&\quad+(\alpha^{2}+2\alpha)\int_{\sn}h^{2}|X|^{2\alpha-2}\sum_{i}\lambda_{i}\sigma^{ii}_{k+1}h^{2}_{i}d\sigma\\ &\leq (\alpha^{2}+\beta)\int_{\sn}|X|^{2\alpha-1}h\langle \nabla h, \nabla |X|\rangle dV_{k}+2\alpha\int_{\sn}|X|^{2\alpha-2}h^{2}\sum_{i}\lambda_{i}\sigma^{ii}_{k+1}h^{2}_{i}d\sigma.
\end{split}
\end{equation}
Note that $|X|\langle \nabla h, \nabla |X|\rangle=\sum_{i}\lambda_{i}h^{2}_{i}\geq c |\nabla h|^{2}$ where $c>0$ depends on $K$. Given \eqref{ee}, we analyze the following two cases.

(Case i) When $\beta\leq 0$ and $p\geq 1-k$ with at least one of these inequalities being strict. In this case, choosing $\alpha=0$, from  \eqref{ee}, we get $\nabla h\equiv0$ on $\sn$. Hence, $\partial K$ is an origin-centred sphere.

(Case ii) When $\beta>0$ and $p>1-k$. By a direct computation,
\begin{equation}
\begin{split}
\label{v1}
&\int_{\sn}|X|^{2\alpha-1}h\langle \nabla h, \nabla |X|\rangle dV_{k}=\int_{\sn}|X|^{2\alpha-2}h\sum_{i}\lambda_{i}h^{2}_{i}dV_{k}\\
&\leq \int_{\sn}|X|^{2\alpha-2}h|\nabla h|^{2}(\Delta h+nh)dV_{k}\\
&=\int_{\sn}|X|^{2\alpha-2}h|\nabla h|^{2}\Delta h dV_{k}+n\int_{\sn}|X|^{2\alpha-2}h^{2}|\nabla h|^{2}dV_{k}.
\end{split}
\end{equation}
Recall that $dV_{k}=h^{p}|X|^{-\beta}d\sigma$. Suppose that $\alpha-\frac{\beta}{2}\geq 0$. By integration by parts, we find
\begin{equation}
\begin{split}
\label{b1}
&\int_{\sn}|X|^{2\alpha-2}h|\nabla h|^{2}\Delta h dV_{k}=\int_{\sn}|X|^{2\alpha-\beta-2}h^{p+1}|\nabla h|^{2}\Delta h d\sigma\\
&=-(p+1)\int_{\sn}|X|^{2\alpha-\beta-2}h^{p}|\nabla h|^{4}d\sigma-\int_{\sn}|X|^{2\alpha-\beta-2}h^{p+1}\langle \nabla h, \nabla |\nabla h|^{2}\rangle d \sigma\\
&\quad -\int_{\sn}\frac{2\alpha-\beta-2}{2}|X|^{2\alpha-\beta-4}h^{p+1}|\nabla h|^{2}\langle \nabla h, \nabla |X|^{2}\rangle d\sigma\\
&=-(p+1)\int_{\sn}|X|^{2\alpha-\beta-2}h^{p}|\nabla h|^{4}d\sigma+2\int_{\sn}|X|^{2\alpha-\beta-2}h^{p+2}|\nabla h|^{2}d\sigma\\
&\quad -\int_{\sn}\left(|X|^{2}+\frac{2\alpha-\beta-2}{2}|\nabla h|^{2}\right) |X|^{2\alpha-\beta-4}h^{p+1} \langle \nabla h, \nabla |X|^{2}\rangle   d\sigma\\
&\leq -(p+1)\int_{\sn}|X|^{2\alpha-\beta-2}h^{p}|\nabla h|^{4}d\sigma+2\int_{\sn}|X|^{2\alpha-\beta-2}h^{p+2}|\nabla h|^{2}d\sigma\\
&\quad -\int_{\sn}\left(\frac{2\alpha-\beta}{2}|\nabla h|^{2} |X|^{2\alpha-\beta-4}h^{p+1} \langle \nabla h, \nabla |X|^{2}\rangle   \right)d\sigma\\
&\leq -(p+1)\int_{\sn}|X|^{2\alpha-\beta-2}h^{p}|\nabla h|^{4}d\sigma+2\int_{\sn}|X|^{2\alpha-\beta-2}h^{p+2}|\nabla h|^{2}d\sigma\\
&= -(p+1)\int_{\sn}|X|^{2\alpha-2}|\nabla h|^{4}dV_{k}+2\int_{\sn}|X|^{2\alpha-2}h^{2}|\nabla h|^{2}dV_{k}.
\end{split}
\end{equation}
Thus from \eqref{v1} and \eqref{b1}, we have
\begin{equation}
\begin{split}
\label{zz}
&(\alpha^{2}+\beta)\int_{\sn}|X|^{2\alpha-1}h\langle \nabla h, \nabla |X|\rangle dV_{k}\\
&\leq -(p+1)(\alpha^{2}+\beta)\int_{\sn}|X|^{2\alpha-2}|\nabla h|^{4}dV_{k} +(\alpha^{2}+\beta)(n+2)\int_{\sn}|X|^{2\alpha-2}h^{2}|\nabla h|^{2}dV_{k}.
\end{split}
\end{equation}
Since $\alpha>0$, by the above argument, we also have
\begin{equation}
\begin{split}
\label{xx}
&2\alpha\int_{\sn}|X|^{2\alpha-2}h^{2}\sum_{i}\lambda_{i}\sigma^{ii}_{k+1}h^{2}_{i}d\sigma\leq 2\alpha(n-k) \int_{\sn}|X|^{2\alpha-2}h|\nabla h|^{2}(\Delta h+nh)dV_{k}\\
&\leq -2\alpha(n-k)(p+1)\int_{\sn}|X|^{2\alpha-2}|\nabla h|^{4}dV_{k}+2\alpha(n+2)(n-k)\int_{\sn}|X|^{2\alpha-2}h^{2}|\nabla h|^{2}dV_{k}.
\end{split}
\end{equation}
Applying \eqref{zz} and \eqref{xx} into \eqref{ee}, we have
\begin{equation}
\begin{split}
\label{c}
&2\int_{\sn}|X|^{2\alpha}h\sum_{i}\lambda_{i}\sigma^{ii}_{k}h^{2}_{i}d\sigma
+(\alpha^{2}+2\alpha)\int_{\sn}h^{2}|X|^{2\alpha-2}\sum_{i}\lambda_{i}\sigma^{ii}_{k+1}h^{2}_{i}d\sigma\\ &\leq [-(k+p-1)-(p+1)(\alpha^{2}+\beta+2\alpha(n-k))]\int_{\sn}|X|^{2\alpha-2}|\nabla h|^{4}dV_{k}\\
&\quad  + [-(k+p-1)+(n+2)(\alpha^{2}+\beta+2\alpha(n-k))]\int_{\sn}|X|^{2\alpha-2}h^{2}|\nabla h|^{2}dV_{k}.
\end{split}
\end{equation}
Note that $-(p+1)\leq (n+2)$. Choose $\alpha =\frac{-2(n-k)+\sqrt{4(n-k)^{2}+4\left( \frac{k+p-1}{n+2}-\beta \right)}}{2} $ with $0< \beta \leq 2k-2n-2+2\sqrt{(n-k+1)^{2}+\frac{k+p-1}{n+2}}$ so that $\alpha-\frac{\beta}{2}\geq 0$ and
\[
(n+2)(\alpha^{2}+\beta+2\alpha(n-k))=k+p-1.
\]
Then \eqref{c} becomes
\begin{equation*}
\begin{split}
&0\leq 2\int_{\sn}|X|^{2\alpha}h\sum_{i}\lambda_{i}\sigma^{ii}_{k}h^{2}_{i}d\sigma
+(\alpha^{2}+2\alpha)\int_{\sn}h^{2}|X|^{2\alpha-2}\sum_{i}\lambda_{i}\sigma^{ii}_{k+1}h^{2}_{i}d\sigma\\ &\leq \left(-(k+p-1)-(p+1)\frac{k+p-1}{n+2}\right)\int_{\sn}|X|^{2\alpha-2}|\nabla h|^{4}dV_{k}\leq 0.
\end{split}
\end{equation*}
Therefore, $\nabla  h\equiv0$ on $\sn$.

\end{proof}
\begin{rem}
If $p>1-k$, one can see that
\[
k+1+2k-2n-2+2\sqrt{(n-k+1)^{2}+\frac{k+p-1}{n+2}}>k+1.
\]
\end{rem}

\section*{Acknowledgment} We would like to thank Yingxiang Hu and Mohammad N. Ivaki for their many useful comments and suggestions on various drafts of this work. The authors also thank the referees for detailed reading and comments that are both  thorough and insightful.

{\bf Conflict of interest:} The authors declare that they have no conflict of interest.

{\bf Data availability:} Data sharing was not applicable to this article and no datasets were generated or analyzed during the current study.


\begin{thebibliography}{20}

\bibitem{A39}
    A. D. Aleksandrov, On the surface area measure of convex bodies, Mat. Sb. (N.S.) {\bf 6} (1939), 167--174.

    \bibitem{A42}
    A. D. Aleksandrov, On the theory of mixed volumes. III. Extensions of two theorems of Minkowski on convex polyhedra to arbitrary convex bodies, Mat. Sb. {\bf 3} (1938), 27--46.

\bibitem{AN97}
B. Andrews, Monotone quantities and unique limits for evolving convex hypersurfaces, Internat. Math. Res. Notices {\bf 1997}, no.~20, 1001--1031.

\bibitem{AB20}
B. Andrews, B. Chow, C. Guenther and \ and\ M. Langford, {\it Extrinsic geometric flows}, Graduate Studies in Mathematics, 206, Amer. Math. Soc., Providence, RI, [2020] \copyright 2020.



\bibitem{Ba84}
J.~M. Ball, Differentiability properties of symmetric and isotropic functions, Duke Math. J. {\bf 51} (1984), no.~3, 699--728.


\bibitem{Be96}
C. Berg, Corps convexes et potentiels sph\'{e}riques, Mat.-Fys. Medd. Danske Vid. Selsk. {\bf 37} (1969), no.~6, 64 pp. (1969).


\bibitem{B19}
G. Bianchi, K. J. B\"{o}r\"{o}czky, A. Colesanti, D. Yang, The $L_p$-Minkowski problem for $-n<p<1$, Adv. Math. {\bf 341} (2019), 493--535.

  \bibitem{BF19}
    K. J. B\"{o}r\"{o}czky\ and\ F. Fodor, The $L_p$ dual Minkowski problem for $p>1$ and $q>0$, J. Differential Equations {\bf 266} (2019), no.~12, 7980--8033.

\bibitem{BLYZ13}
       K. J. B\"{o}r\"{o}czky, E. Lutwak, D. Yang\ and \  G. Zhang, The log-Brunn-Minkowski inequality, Adv. Math. {\bf 231} (2012), no.~3-4, 1974--1997.

   \bibitem{BLYZ12}
       K. J. B\"{o}r\"{o}czky, E. Lutwak, D. Yang\ and \  G. Zhang, The logarithmic Minkowski problem, J. Amer. Math. Soc. {\bf 26} (2013), no.~3, 831--852.

        \bibitem{BLYZ19}
     K. J. B\"{o}r\"{o}czky, E. Lutwak, D. Yang,  G. Zhang\ and \ Y. Zhao, The dual Minkowski problem for symmetric convex bodies, Adv. Math. {\bf 356} (2019), 106805, 30 pp.



\bibitem{BKMZ24}
     K. J. B\"{o}r\"{o}czky, \'{A}. Kov\'{a}cs, S. Mui\ and \  G. Zhang, Dual curvature density equation with group symmetry, arXiv: 2503.10044.



       \bibitem{BS24}
K. J. B\"{o}r\"{o}czky\ and\ C. Saroglou, Uniqueness when the $L_p$ curvature is close to be a constant for $p\in [0,1)$, Calc. Var. Partial Differential Equations {\bf 63} (2024), no.~6, Paper No. 154, 26 pp.


    \bibitem{B17}
K. J. B\"{o}r\"{o}czky\ and\ H. T. Trinh, The planar $L_p$-Minkowski problem for $0<p<1$, Adv. in Appl. Math. {\bf 87} (2017), 58--81.


\bibitem{BCD17}
S. Brendle, K. Choi\ and\ P. Daskalopoulos, Asymptotic behavior of flows by powers of the Gaussian curvature, Acta Math. {\bf 219} (2017), no.~1, 1--16.


\bibitem{BIS23a}
P. Bryan, M.~N. Ivaki\ and\ J. Scheuer, Christoffel-Minkowski flows, Trans. Amer. Math. Soc. {\bf 376} (2023), no.~4, 2373--2393.


\bibitem{BIS23b}
P. Bryan, M.~N. Ivaki\ and\ J. Scheuer, Constant rank theorems for curvature problems via a viscosity approach, Calc. Var. Partial Differential Equations {\bf 62} (2023), no.~3, Paper No. 98, 19 pp.

\bibitem{CHZ19}
    C. Chen, Y. Huang\ and\ Y. Zhao, Smooth solutions to the $L_p$ dual Minkowski problem, Math. Ann. {\bf 373} (2019), no.~3-4, 953--976.

\bibitem{CCL21}
H. Chen, S. Chen\ and\ Q. Li, Variations of a class of Monge-Amp\`ere-type functionals and their applications, Anal. PDE {\bf 14} (2021), no.~3, 689--716.


\bibitem{CL21}
H. Chen\ and\ Q. Li, The $L_ p$ dual Minkowski problem and related parabolic flows, J. Funct. Anal. {\bf 281} (2021), no.~8, Paper No. 109139, 65 pp.


\bibitem{CFL22}
S. Chen, Y. Feng\ and\ W. Liu, Uniqueness of solutions to the logarithmic Minkowski problem in $\mathbb R^3$, Adv. Math. {\bf 411} (2022), part A, Paper No. 108782, 18 pp.



\bibitem{CY20}
S. Chen, Y. Huang, Q. Li\ and\ J. Liu, The $L_ p$-Brunn-Minkowski inequality for $p<1$, Adv. Math. {\bf 368} (2020), 107166, 21 pp.


\bibitem{CL17}
S. Chen, Q. Li\ and\ G. Zhu, On the $L_p$ Monge-Amp\`ere equation, J. Differential Equations {\bf 263} (2017), no.~8, 4997--5011.

\bibitem{CTX25}
X. Chen, Q. Tu\ and\ N. Xiang, The $L_{p}$ dual Christoffel-Minkowski problem for the case $p\geq q$, arXiv: 2503.01454.


\bibitem{CY76}
S. Y. Cheng\ and\ S. T. Yau, On the regularity of the solution of the $n$-dimensional Minkowski problem, Comm. Pure Appl. Math. {\bf 29} (1976), no.~5, 495--516.


\bibitem{CW06}
K.-S. Chou\ and\ X.~J. Wang, The $L_p$-Minkowski problem and the Minkowski problem in centroaffine geometry, Adv. Math. {\bf 205} (2006), no.~1, 33--83.

\bibitem{Cr85}
E. B. Christoffel, Ueber die Bestimmung der Gestalt einer krummen Oberfl\"{a}che durch lokale Messungen auf derselben, J. Reine Angew. Math. {\bf 64} (1865), 193--209.




\bibitem{DL23}
S. Ding\ and\ G.~H. Li, A class of inverse curvature flows and $L^p$ dual Christoffel-Minkowski problem, Trans. Amer. Math. Soc. {\bf 376} (2023), no.~1, 697--752.

\bibitem{F67}
W.~J. Firey, The determination of convex bodies from their mean radius of curvature functions, Mathematika {\bf 14} (1967), 1--13.


\bibitem{F68}
W.~J. Firey, Christoffel's problem for general convex bodies, Mathematika {\bf 15} (1968), 7--21.

\bibitem{Fir70}  W.~J. Firey, Intermediate Christoffel-Minkowski problems for figures of revolution, Israel J. Math. \textbf{8} (1970): 384--390.

\bibitem{G06}
R. J. Gardner, {\it Geometric tomography}, second edition, Encyclopedia of Mathematics and its Applications, 58, Cambridge University Press, New York, 2006.



 \bibitem{Gu23}
P. Guan, A weighted gradient estimate for solutions of $L^p$ Christoffel-Minkowski problem, Math. Eng. {\bf 5} (2023), no.~3, Paper No. 067, 14 pp.



\bibitem{GM03}
P. Guan\ and\ X.-N. Ma, The Christoffel-Minkowski problem. I. Convexity of solutions of a Hessian equation, Invent. Math. {\bf 151} (2003), no.~3, 553--577.

\bibitem{GX18}
P. Guan\ and\ C. Xia, $L^p$ Christoffel-Minkowski problem: the case $1<p<k+1$, Calc. Var. Partial Differential Equations {\bf 57} (2018), no.~2, Paper No. 69, 23 pp.

\bibitem{GLM06}
P. Guan, C.-S. Lin\ and\ X.~N. Ma, The Christoffel-Minkowski problem. II. Weingarten curvature equations, Chinese Ann. Math. Ser. B {\bf 27} (2006), no.~6, 595--614.

\bibitem{GMZ06}
P. Guan, X.~N. Ma\ and\ F. Zhou, The Christofel-Minkowski problem. III. Existence and convexity of admissible solutions, Comm. Pure Appl. Math. {\bf 59} (2006), no.~9, 1352--1376.

\bibitem{GRW15}
P. Guan, C.~Y. Ren\ and\ Z. Wang, Global $C^2$-estimates for convex solutions of curvature equations, Comm. Pure Appl. Math. {\bf 68} (2015), no.~8, 1287--1325.

\bibitem{HMS24}
C.~Q. Hu, X.~N. Ma\ and\ C.~L. Shen, On the Christoffel-Minkowski problem of Firey's $p$-sum, Calc. Var. Partial Differential Equations {\bf 21} (2004), no.~2, 137--155.

\bibitem{HJ24}
J. Hu, The Dual Minkowski Problem for Positive Indices, Int. Math. Res. Not. IMRN {\bf 2025}, no.~13, rnaf192.


\bibitem{HI242}
Y. Hu\ and\ M.~N. Ivaki, On the uniqueness of solutions to the isotropic $L_p$ dual Minkowski problem, Nonlinear Anal. {\bf 241} (2024), Paper No. 113493, 6 pp.


\bibitem{HI24}
Y. Hu\ and\ M.~N. Ivaki, Prescribed $L_p$ curvature problem, Adv. Math. {\bf 442} (2024), Paper No. 109566, 15 pp.


\bibitem{HI25}
Y. Hu\ and\ M.~N. Ivaki, Stability of the Cone-Volume Measure With Near Constant Density, Int. Math. Res. Not. IMRN {\bf 2025}, no.~6, rnaf062.


\bibitem{HLYZ16}
Y. Huang, E. Lutwak, D. Yang\ and\ G. Zhang, Geometric measures in the dual Brunn-Minkowski theory and their associated Minkowski problems, Acta Math. {\bf 216} (2016), no.~2, 325--388.


   \bibitem{HZ18}
    Y. Huang\ and\ Y. Zhao, On the $L_p$ dual Minkowski problem, Adv. Math. {\bf 332} (2018), 57--84.


%\bibitem{Hil53}
%D. Hilbert, {\it Grundz\"{u}ge einer allgemeinen Theorie der linearen Integralgleichungen}, Chelsea, New York, NY, 1953.


%\bibitem{Hu02}
%A. Hurwitz, Sur quelques applications g\'{e}om\'{e}triques des s\'{e}ries de Fourier, Ann. Sci. \'{E}cole Norm. Sup. (3) {\bf 19} (1902), 357--408.

\bibitem{Iva19}
M. N. Ivaki, Deforming a hypersurface by principal radii of curvature and support function, Calc. Var. Partial Differential Equations {\bf 58} (2019), no.~1, Paper No. 1, 18 pp.

\bibitem{IM23}
M.~N. Ivaki\ and\ E. Milman, Uniqueness of solutions to a class of isotropic curvature problems, Adv. Math. {\bf 435} (2023), part A, Paper No. 109350, 11 pp.

\bibitem{IM25}
M.~N. Ivaki\ and\ E. Milman, $L^p$-Minkowski problem under curvature pinching, Int. Math. Res. Not. IMRN {\bf 2024}, no.~10, 8638--8652.


\bibitem{JW17}
    Y. Jiang\ and\ Y. Wu, On the 2-dimensional dual Minkowski problem, J. Differential Equations {\bf 263} (2017), no.~6, 3230--3243.




\bibitem{KM22}
A.~V. Kolesnikov\ and\ E. Milman, Local $L^p$-Brunn-Minkowski inequalities for $p<1$, Mem. Amer. Math. Soc. {\bf 277} (2022), no.~1360, v+78 pp.


\bibitem{LW22}
H. Li\ and\ Y. Wan, Classification of solutions for the planar isotropic $L_{p}$ dual Minkowski problem, 2022, arXiv: 2209.14630.


\bibitem{LW24}
H. Li\ and\ Y. Wan, Uniqueness of solutions to some classes of anisotropic and isotropic curvature problems, J. Funct. Anal. {\bf 287} (2024), no.~3, Paper No. 110471, 30 pp.


\bibitem{LRW16}
M. Li, C.~Y. Ren\ and\ Z. Wang, An interior estimate for convex solutions and a rigidity theorem, J. Funct. Anal. {\bf 270} (2016), no.~7, 2691--2714.


\bibitem{LLJ22}
  Q. Li, J. Liu\ and\ J. Lu, Nonuniqueness of solutions to the $L_p$ dual Minkowski problem, Int. Math. Res. Not. IMRN {\bf 2022}, no.~12, 9114--9150.


   \bibitem{LSW20}
  Q. Li, W. Sheng\ and\ X.-J. Wang, Flow by Gauss curvature to the Aleksandrov and dual Minkowski problems, J. Eur. Math. Soc. (JEMS) {\bf 22} (2020), no.~3, 893--923.



\bibitem{LY89}
Y.~Y. Li, Degree theory for second order nonlinear elliptic operators and its applications, Comm. Partial Differential Equations {\bf 14} (1989), no.~11, 1541--1578.


     \bibitem{LW13}
J. Lu\ and\ X.~J. Wang, Rotationally symmetric solutions to the $L_p$-Minkowski problem, J. Differential Equations {\bf 254} (2013), no.~3, 983--1005.


 \bibitem{LYZ04}
    E. Lutwak, D. Yang\ and\ G. Zhang, On the $L_p$-Minkowski problem, Trans. Amer. Math. Soc. {\bf 356} (2004), no.~11, 4359--4370.




%\bibitem{LL22}
%B. Li, H. Ju\ and\ Y. Liu, A flow method for a generalization of $L_p$ Christofell-Minkowski problem, Commun. Pure Appl. Anal. {\bf 21} (2022), no.~3, 785--796.





\bibitem{L93}
    E. Lutwak, The Brunn-Minkowski-Firey theory. I. Mixed volumes and the Minkowski problem, J. Differential Geom. {\bf 38} (1993), no.~1, 131--150.

  \bibitem{LO95}  E. Lutwak\ and  V. Oliker, On the regularity of solutions to a generalization of the Minkowski problem,  J. Differ. Geom. \textbf{41} (1995): 227--246.

\bibitem{LYZ18}
E. Lutwak, D. Yang\ and\ G. Zhang, $L_p$ dual curvature measures, Adv. Math. {\bf 329} (2018), 85--132.



\bibitem{MWW24}
X. Mei, G. Wang\ and\ L. Weng, Prescribed $L_{p}$ quotient curvature problem and related eigenvalue problem, arxiv: 2402.12314.

%\bibitem{ML22}
%E. Milman, Centro-affine differential geometry and the log-Minkowski problem, J. Eur. Math. Soc. (JEMS) {\bf 27} (2025), no.~2, 709--772.


\bibitem{M897}
      H. Minkowski, Allgemeine Lehrs\"{a}tze \"{u}ber die convexen Polyeder, Nachr. Ges. Wiss. G\"{o}ttingen (1897), 198--219.

    \bibitem{M903}
    H. Minkowski, Volumen und Oberfl\"{a}che, Math. Ann. {\bf 57} (1903), no.~4, 447--495.


     \bibitem{Mu24}
    S. Mui, On the $L^p$ dual Minkowski problem for $-1<p<0$, Calc. Var. Partial Differential Equations {\bf 63} (2024), no.~8, Paper No. 215, 19 pp.



\bibitem{N53}
L. Nirenberg, The Weyl and Minkowski problems in differential geometry in the large, Comm. Pure Appl. Math. {\bf 6} (1953), 337--394.


\bibitem{P78}
A. Pogorelov, {\it The Minkowski multidimensional problem}, translated from the Russian by Vladimir Oliker, Scripta Series in Mathematics, Winston, Washington, DC, 1978.


\bibitem{S14}
R. Schneider, {\it Convex bodies: the Brunn-Minkowski theory}, second expanded edition, Encyclopedia of Mathematics and its Applications, 151, Cambridge University Press, Cambridge, 2014.

\bibitem{SY20}
W. Sheng\ and\ C. Yi, A class of anisotropic expanding curvature flows, Discrete Contin. Dyn. Syst. {\bf 40} (2020), no.~4, 2017--2035.

%\bibitem{S33}
%W. S\"{u}ss, Bestimmung einer geschlossenen konvexen Fl\"{a}che durch die Summe ihrer Hauptkr\"{u}mmungsradien, Math. Ann. {\bf 108} (1933), no.~1, 143--148.


\bibitem{Z17}
    Y. Zhao, The dual Minkowski problem for negative indices, Calc. Var. Partial Differential Equations {\bf 56} (2017), no.~2, Paper No. 18, 16 pp.

    \bibitem{Z18}
    Y. Zhao, Existence of solutions to the even dual Minkowski problem, J. Differential Geom. {\bf 110} (2018), no.~3, 543--572.


\bibitem{Zhu14}
    G. Zhu, The logarithmic Minkowski problem for polytopes, Adv. Math. {\bf 262} (2014), 909--931.

\bibitem{Zh15}
G. Zhu, The $L_p$ Minkowski problem for polytopes for $0<p<1$, J. Funct. Anal. {\bf 269} (2015), no.~4, 1070--1094.



    \bibitem{Zhu15}
    G. Zhu, The centro-affine Minkowski problem for polytopes, J. Differential Geom. {\bf 101} (2015), no.~1, 159--174.




\end{thebibliography}
\end{document}